\documentclass[final]{siamonline190516}

\usepackage{amsfonts,amsmath,amssymb}
\usepackage{graphicx,subcaption,caption}
\usepackage{epstopdf}
\usepackage{algpseudocode}
\usepackage{siunitx}
\usepackage{cleveref}
\usepackage{tabularx,booktabs}
\ifpdf
  \DeclareGraphicsExtensions{.eps,.pdf,.png,.jpg}
\else
  \DeclareGraphicsExtensions{.eps}
\fi

\graphicspath{{./Figs/}}

\usepackage{enumitem}
\setlist[enumerate]{leftmargin=.5in}
\setlist[itemize]{leftmargin=.5in}

\newsiamremark{remark}{Remark}
\newsiamremark{hypothesis}{Hypothesis}
\crefname{hypothesis}{Hypothesis}{Hypotheses}
\newsiamthm{claim}{Claim}

\newcommand{\diff}{\ensuremath{\mathrm{d}}}
\newcommand{\liediff}{\ensuremath{\mathfrak{d}}}
\newcommand{\fourier}{\ensuremath{\mathcal{F}}}
\newcommand{\real}{\ensuremath{\mathbb{R}}}
\newcommand{\complex}{\ensuremath{\mathbb{C}}}
\newcommand{\SO}[1]{\ensuremath{\mathrm{SO}(#1)}}
\newcommand{\SE}[1]{\ensuremath{\mathrm{SE}(#1)}}
\newcommand{\sph}{\ensuremath{\mathrm{S}}}
\newcommand{\torus}{\ensuremath{\mathrm{T}}}

\newcommand{\algrule}[1][.2pt]{\par\vskip.2\baselineskip\hrule height #1\par\vskip.2\baselineskip}

\newcommand{\makecell}[2][@{}c@{}]{\begin{tabular}{#1}#2\end{tabular}}

\headers{Uncertainty Propagation for Stochastic Hybrid Systems}{W. Wang and T. Lee}

\title{Uncertainty Propagation for General Stochastic Hybrid Systems\\ on Compact Lie Groups\thanks{Submitted to the editors DATE.
\funding{This work has been supported in part by AFOSR under the grant FA9550-18-1-0288.}}}

\author{Weixin Wang\thanks{The George Washington University, DC 
  (\email{wwang442@gwu.edu, tylee@gwu.edu}).}
\and Taeyoung Lee\footnotemark[2]}

\usepackage{amsopn}
\DeclareMathOperator{\diag}{diag}

\ifpdf
\hypersetup{
  pdftitle={An Example Article},
  pdfauthor={D. Doe, P. T. Frank, and J. E. Smith}
}
\fi

\begin{document}

\maketitle

\begin{abstract}
	This paper deals with uncertainty propagation of general stochastic hybrid systems (GSHS) where the continuous state space is a compact Lie group.
	A computational framework is proposed to solve the Fokker-Planck (FP) equation that describes the time evolution of the probability density function for the state of GSHS.
	The FP equation is split into two parts: the partial differential operator corresponding to the continuous dynamics, and the integral operator arising from the discrete dynamics.
	These two parts are solved alternatively using the operator splitting technique.
	Specifically, the partial differential equation is solved by the spectral method where the density function is decomposed into a linear combination of a complete orthonormal function basis brought forth by the Peter-Weyl theorem, thereby resulting an ordinary differential equation.
	Next, the integral equation is solved by approximating the integral by a finite summation using a quadrature rule.
	The proposed method is then applied to a three-dimensional rigid body pendulum colliding with a wall, evolving on the product of the three-dimensional special orthogonal group and the Euclidean space.
	It is  illustrated that the proposed method exhibits numerical results consistent with a Monte Carlo simulation, while explicitly generating the density function that carries the complete stochastic information of the hybrid state.
\end{abstract}



\section{Introduction}
General stochastic hybrid system (GSHS) is a stochastic dynamical system that exhibits both continuous and discrete random behaviors~\cite{bujorianu2006toward}.
In a GSHS, the hybrid state consists of two parts: the continuous state that takes the value on a smooth manifold, and the discrete state that lies on a countable set. 
The continuous dynamics is defined by stochastic differential equations (SDEs) indexed by the discrete state, describing the evolution of continuous state between jumps.
The discrete dynamics describes the stochastic jump of the state, which is triggered by a Poisson process with a state-dependent rate function.
The uncertainty after the jump is represented by a stochastic kernel.
GSHS exhibits rich dynamics caused by the interplay between the continuous state and the discrete counterpart, and it has been used to model various complex systems, such as chemical reactions \cite{hespanha2005stochastic}, neuron activities \cite{pakdaman2010fluid}, air traffic control \cite{blom2009rare,prandini2008application,seah2009stochastic}, and communication networks \cite{hespanha2004stochastic}.

Uncertainty propagation involves advecting a probability density along the flow of a dynamical system according to the Fokker--Planck (FP) equation. 
The probability density can be approximated by, for example, the first $n$-moments~\cite{KumSinAASCE06}, which leads to Monte Carlo methods~\cite{MetUlaJASA49,HamHan75}, Gaussian closure methods~\cite{IyeDasJAM78,LevMorSJAM98}, and equivalent linearization and stochastic averaging~\cite{RobSpa03,PraIJNM01}. 
But, Monte Carlo methods do not propagate the probability density function directly. 
Other methods involve low-order approximations of the dynamical system, which are suitable only for moderately nonlinear systems as the omitted higher-order terms can lead to significant errors, particularly for long time intervals.
For stochastic hybrid systems, uncertainty propagation has been focused on the case when the continuous state lies in the Euclidean space.
For example, the interacting multiple model approach \cite{blom1988interacting} and the salted Kalman filter \cite{kong2021salted} linearize the dynamics and use the Gaussian distribution to describe uncertainties. 
In \cite{blom2004particle,tafazoli2006hybrid}, particle filters are employed to propagate random samples through the dynamics to approximate the uncertainty distribution.
Alternatively, to propagate the probability density function directly, the FP equation has been extended for GSHS into integro-partial differential equations (IPDEs) \cite{bect2010unifying,hespanha2004stochastic}.
And it has been solved using finite difference method \cite{liu2014hybrid} and spectral method \cite{wang2020spectral}.

In this paper, we study the uncertainty propagation for GSHS whose continuous state evolves on a compact Lie group $G$.
More specifically, given an initial probability distribution of the state, we wish to construct the probability distribution at an arbitrary time through GSHS, by solving the corresponding FP equation represented by IPDEs on $G$.
To address the presence of partial differentiation and integration in the FP equation, we employ the operator splitting method \cite{macnamara2016operator}.
Specifically, the FP equation is decomposed into two parts: the continuous dynamics which only contains the partial differential operator, and the discrete dynamics which only contains the integral operator.
These two individual equations are solved alternatively over a small time step using their respective numerical methods, and they are combined by a first order splitting scheme.

For the partial differential equation corresponding to the continuous dynamics, we use the classic spectral method.
The spectral method has been used to solve the FP equations on $\SE{2}$ and $\SO{3}$ \cite{lee2015stochastic,lee2008global,wang2006solving} for uncertainty propagation of stochastic dynamical systems without discrete dynamics. 
This utilizes the Peter-Weyl theorem \cite{PetWeyMA27}, which states that the matrix components of all finite dimensional irreducible unitary representations of a compact Lie group form a complete orthonormal basis for the space of square integrable functions.
As such, an arbitrary probability density function on $G$ can be approximated by a linear combination of the matrix elements of irreducible unitary representations.
Further using the operational properties of the representation, the FP equation is transformed into ordinary differential equations (ODEs) of the coefficients, which can be integrated by standard ODE solvers. 
Next, the integro-differential equation corresponding to the discrete dynamics is approximated by a quadrature rule over a grid, such that the density values on the grid are propagated by another set of ODEs. 
A useful property is that the grid for the discrete dynamics can be selected to be compatible with the harmonic analysis for the continuous dynamics so as to improve the computational efficiency of the overall splitting scheme. 

Compared to conventional methods based on Gaussian distributions or their mixtures, the proposed method has the advantage of being non-parametric, i.e., it does not assume a specific family of distributions, but applies to density functions with arbitrary shapes.
The proposed method constructs a probability density function, which carries the complete stochastic information about the propagated state, and as such, it can be directly used for visualization or calculating descriptive measures, such as moments, number and locations of local maxima, etc.
In this regard, although the Monte Carlo method is also non-parametric, the information of the state is implicitly carried by random samples, which is usually hard to be distilled into usable forms other than calculating moments, especially when the number of samples is large.
Also, the Monte Carlo method cannot deal with large uncertainties efficiently \cite{wang2020spectral}.
The downside of the proposed approach is that as a spectral method, its computational complexity increases exponentially with the dimension of continuous space, and quickly becomes infeasible \cite{tadmor2012review}.

In short, the main contribution of this paper is the computational framework to propagate uncertainties though GSHS on a compact Lie gorup. 
The use of noncommutative harmonic analysis to represent the uncertainty distribution in a global fashion overcomes a fundamental limitation of existing techniques, which implicitly assume that the uncertainty is localized, or has a canonical form.
By solving the Fokker--Planck equation directly, the probability density that describes the complete stochastic properties of a hybrid system is propagated. 

The rest of this paper is organized as follows: Section \ref{sec:problem} reviews the formulation of GSHS considered in this paper, and introduces its associated FP equation.
The proposed algorithm for uncertainty propagation is introduced in Section \ref{sec:UP} when the continuous state space is a general compact Lie group.
In Section \ref{sec:example}, we focus on a specific example of a 3D pendulum colliding with a wall, where the continuous state space is $\SO{3}\times \real^2$.


\section{Problem Formulation} \label{sec:problem}

In this section, we give a formal definition \cite{bujorianu2006toward} of the GSHS considered in this paper, and introduce the corresponding FP equation that describes the evolution of the probability density function over time.

\subsection{General Stochastic Hybrid System} \label{sec:GSHS}
The GSHS considered in this paper is defined as a collection $H = \{X,a,b,Init,\lambda,K\}$ as follows:
\begin{itemize}
	\item $X = G\times S$ is the hybrid state space, where $G$ is a $N_g$-dimensional compact Lie group, and $S$ is a set composed of $N_s$ discrete modes.
	The hybrid state is denoted by $(g,s) \in G\times S$.
	\item $Init\ :\ \mathcal{B}(X) \to [0,1]$ is the initial uncertainty distribution of the hybrid state, where $\mathcal{B}(X)$ is all Borel sets in $X$.
	\item The continuous state evolves according to the following stochastic differential equations between discrete jumps:
	\begin{equation} \label{eqn:SDE}
		g^{-1} \diff g = a(t,g,s)^\wedge\diff t + (b(t,s)\diff W_t)^\wedge 
	\end{equation}
	where $a\, :\, \real\times X \to \real^{N_g}$ is the drifting vector field, and $b\, :\, \real\times S \to \real^{N_g\times N_w}$ is the coefficient matrix for diffusion.
    Next, $W_t$ is a $N_w$-dimensional standard Wiener process.
	The map $(\cdot)^\wedge\ :\ \real^{N_g} \to \mathfrak{g}$ is the natural identification of $\real^{N_g}$ and $\mathfrak{g}$, the Lie algebra of $G$.
	Since $b$ does not depend on $g$, \eqref{eqn:SDE} can be defined either in Ito's or Stratonovich's sense.
	\item The discrete jump is triggered by a Poisson process, with a rate function $\lambda:\ \: X \to \real^+$ dependent on the hybrid state.
	\item During each discrete jump, the hybrid state is reset according to a stochastic kernel $K\ :\ (X,\mathcal{B}(X)) \to [0,1]$, such that $K(x^-,X^+)$ is the probability of $x^-\in X$ being reset into the set $X^+\in \mathcal{B}(X)$.
\end{itemize}

One restriction of the GSHS defined above is that it does not allow the discrete jump to be triggered by the continuous state $g$ entering a certain guard set in a deterministic fashion. 
However, such \textit{forced} jumps can be approximated by a Poisson process after choosing the rate function sufficiently large inside the guard set, and zero outside \cite{hespanha2004stochastic}.
This will be illustrated by the 3D pendulum example in Section \ref{sec:example}.

We also assume the initial distribution has a probability density function for each $s\in S$, i.e., $Init(A) = \sum_{s\in S} \int_{(g,s)\in A} p(t_0,g,s)\diff g$ for all $A\in\mathcal{B}(X)$, where $\diff g$ is the bi-invariant Haar measure on $G$ normalized such that $\int_{g\in G}\diff g = 1$.
Furthermore, the discrete transition kernel $K$ can also be written as a set of density functions:
\begin{equation*}
	K(x^-,X^+) = \sum_{s^+\in S} \int_{(g^+,s^+)\in X^+} \kappa(g^-,s^-,g^+,s^+) \diff g^+,
\end{equation*}
where $\kappa: X\times X\rightarrow\real$.

Let $(\Omega,\mathcal{F},\mathbb{P})$ be the underlying probability space, where $\Omega$ is the sample space, $\mathcal{F}$ is a sigma-algebra over $\Omega$, and $\mathbb{P}$ denotes the probability measure on $\mathcal{F}$.
For a given $\omega\in\Omega$, let $\{u_k(\omega)\}$ be a sequence of independent uniformly distributed random variables on $[0,1]$.
Then an execution of the GSHS defined above can be generated according to the following procedure.
\begin{enumerate}
	\item Initialize $g(\omega,t_0)$ and $s(\omega,t_0)$ from the initial distribution $Init$.
	\item Let $t_1(\omega) = \sup\left\{ t\ :\ \exp\left( -\int_{t_0}^t \lambda(g(\omega,\tau),s(\omega,t_0)) \diff\tau \right) > u_1(\omega) \right\}$ be the time of the first jump.
	\item During $t\in[t_0,t_1(\omega))$, $g(\omega,t)$ is a sample path of SDE \eqref{eqn:SDE} with $s=s(\omega,t_0)$, and $s(\omega,t) = s(\omega,t_0)$.
	\item At time $t_1$, the state is reset to $(g(\omega,t_1^+),s(\omega,t_1^+))$ as a sample from the kernel $\kappa(g(\omega,t_1^-),s(\omega,t_0),z^+,s^+)$.
	\item If $t_1 < \infty$, repeat from 2) with $t_0$, $s_0$, $t_1$, $u_1$ replaced by $t_k(\omega)$, $s(\omega,t_k^+)$, $t_{k+1}(\omega)$, $u_{k+1}(\omega)$ for $k = 1,2,\ldots$.
\end{enumerate}

\subsection{Fokker-Planck Equation for GSHS}

The FP equation for GSHS describes how its density function evolves over time \cite{bect2010unifying,hespanha2004stochastic}, and it is given as a set of IPDEs as follows:
\begin{equation} \label{eqn:FP}
	\begin{aligned}
		\frac{\partial p(t,g,s)}{\partial t} = &\underbrace{-\sum_{i=1}^{N_g} \liediff_j\left( a_j(t,g,s)p(t,g,s) \right) + \sum_{j,k=1}^{N_g} D_{j,k}(t,s) \liediff_j\liediff_k p(t,g,s)}_{\mathcal{L}_c^* p(t,g,s)} \\
		&+ \underbrace{\sum_{s^-\in S} \int_{g^-\in G} \kappa(g^-,s^-,g,s) \lambda(g^-,s^-) p(t,g^-,s^-) \diff g^- - \lambda(g,s)p(t,g,s)}_{\mathcal{L}_d^* p(t,g,s)},
	\end{aligned}
\end{equation}
where the subscripts denote the indices of a vector or matrix, and $D = \tfrac{1}{2}bb^T$.
Moreover, $\liediff_j$ is the left-trivialized derivative of a function on $G$, i.e., $\liediff_j f(g) = \tfrac{\diff}{\diff t}\big|_{t=0} f\big( g\exp(t\hat e_j) \big)$, where $\exp\ :\ \mathfrak{g} \to G$ is the exponential map, and $e_j$ is the $j$-th standard base vector of $\real^{N_g}$.
For each $s\in S$, \eqref{eqn:FP} defines an IPDE for $p(t,g,s)$, and thus, there are a total of $N_s$ IPDEs.

The FP equation can be interpreted as follows. 
The first two terms on the right hand side of \eqref{eqn:FP} represent the evolution caused by the continuous process: the first one represents advection due to the drift vector field, and the second corresponds to diffusion caused by noise, or the Wiener process. 
The last two terms of \eqref{eqn:FP} describe the evolution due to discrete jumps.
The third term represents the densities transitioned into $(g,s)$ from $(g^-,s^-)$ before the jump, weighted by how likely the jump happens ($\lambda$), and how likely the density is transitioned into $(g,s)$ ($\kappa$). 
The last term represents the density transitioned out of $(g,s)$. 
To distinguish these two parts explicitly, \eqref{eqn:FP} is written as
\begin{equation}\label{eqn:FP_L}
	\frac{\partial p(t,g,s)}{\partial t} = \mathcal{L}^*_c p(t,g,s) + \mathcal{L}^*_d p(t,g,s),
\end{equation}
where $\mathcal{L}^*_c$ and $\mathcal{L}^*_d$ denote the adjoint of the infinitesimal generators of the continuous SDE, and the discrete jump of the GSHS, respectively.

The FP equation \eqref{eqn:FP} describes the evolution of the probability density along the flows of GSHS on a Lie group. 
In contrast to the Fokker--Planck equation of non-hybrid systems, which is a partial-differential equation, \eqref{eqn:FP} exhibits fundamental challenges, as it is an \textit{integro-partial differential equation} that involves both partial differentiation and integration. 
Another challenge is that the probability density is defined on a nonlinear Lie group $G$, so the existing computational techniques in solving the Fokker--Planck equation on a linear space $\real^n$ cannot be directly applied. 
In this paper, these are addressed by utilizing noncommutative harmonic analysis and the splitting technique. 


\section{Uncertainty Propagation for GSHS} \label{sec:UP}

In this section, we present a computational framework to solve \eqref{eqn:FP_L} via the spectral method using noncommutative harmonic analysis on $G$.
However, the integral term in \eqref{eqn:FP} causes issues in the spectral method as there is no closed formula to express the Fourier coefficients of the integral of a function $f(g)$ over $G$ as the Fourier coefficients of $f$.
Even though a closed formula exists on $\real^N$, it has been shown that taking the Fourier transform of the integral term directly involves intensive computations \cite{wang2020spectral2}.

Instead, we adopt the operator splitting technique, where \eqref{eqn:FP_L} is split into two equations:
\begin{subequations} \label{eqn:FP_split}
	\begin{align}
		\frac{\partial p^c(t,g,s)}{\partial t} &= \mathcal{L}^*_c p^c(t,g,s), \label{eqn:FP_cont} \\
		\frac{\partial p^d(t,g,s)}{\partial t} &= \mathcal{L}^*_d p^d(t,g,s). \label{eqn:FP_dist}
	\end{align}
\end{subequations}
The desirable features are that in the absence of the term $\mathcal{L}^*_d p$, \eqref{eqn:FP_cont} corresponds to a PDE on $G$ for each discrete state; and without $\mathcal{L}^*_cp$, \eqref{eqn:FP_dist} becomes an integro-differential equation without partial differentiation.
These can be addressed by using the spectral method and numerical quadrature respectively.
Then, the solution of each part can be combined with the operator splitting. 

\subsection{Propagation over Continuous Dynamics}

First, we solve \eqref{eqn:FP_cont} via noncommutative harmonic analysis.
The objective is to decompose the density function $p^c(t,g,s)$ into a linear combination of an orthonormal basis of a function space on $G$ for each $t\geq t_0$ and $s\in S$.
Then \eqref{eqn:FP_cont} can be converted to a set of ODEs for the coefficients of the linear combination, which can be solved via standard numerical integration schemes for ODE.

We first summarize harmonic analysis on a compact Lie group $G$ \cite{chirikjian2000engineering,maslen1998efficient}.
Let $U^l(g)\ :\, G \to \mathrm{GL}(d(l),\complex)$ be the $l$-th irreducible unitary representation of $G$ with a finite dimension $d(l)$, and the collection of all $l$ be denoted by $\hat{G}$.
Then by the Peter-Weyl theorem, the functions $\big\{ \{ U^l_{m_1,m_2} \}_{m_1,m_2=1}^{d(l)} \big\}_{l\in\hat{G}}$ form an orthonormal basis of the function space $L^2(G)$.
That is, for any square integrable function $f\, :\, G\to \complex$, it can be decomposed as a linear combination
\begin{equation} \label{eqn:idft_G}
	f(g) = \sum_{l\in \hat{G}} \sum_{m_1,m_2=1}^{d(l)} d(l) \fourier_{m_1,m_2}^l[f] U^l_{m_1,m_2}(g),
\end{equation}
where $\fourier_{m_1,m_2}^l[f]$ are the Fourier coefficients of $f$, given by the $(m_1,m_2)$-th entry of
\begin{equation} \label{eqn:dft_G}
	\fourier^l[f] = \int_{g\in G} f(g) U^l(g^{-1}) \diff g.
\end{equation}
Equation \eqref{eqn:idft_G} and \eqref{eqn:dft_G} are called the inverse and forward Fourier transform of function $f$.

One crucial property is that we can express the Fourier coefficient of the derivative $\fourier[\liediff_j f]$ in terms of $\fourier[f]$ \cite{chirikjian2000engineering}.
Let $u^l\, :\, \mathfrak{g}\to \mathfrak{gl}(d(l),\complex)$ be the associated Lie algebra representation of $U^l$, i.e., for all $X\in\mathfrak{g}$, $u^l(X) = \tfrac{\diff}{\diff t} \big|_{t=0} U^l(\exp(tX))$.
Then, using that $\diff g$ is invariant under translation, we have
\begin{equation} \label{eqn:dF_G}
	\begin{aligned}
		\fourier^l[\liediff_j f] &= \lim\limits_{t\to 0} \left( \int_{g\in G} f(g\exp(t\hat{e}_j)) U^l(g^{-1}) \diff g - \int_{g\in G} f(g) U^l(g^{-1}) \diff g \right) / t \\
		&= \lim\limits_{t\to 0} \left( \int_{g\in G} f(g) U^l\left( \exp(t\hat{e}_j)g^{-1} \right)\diff g - \int_{g\in G} f(g) U^l(g^{-1}) \diff g \right) / t \\
		&= \lim\limits_{t\to 0} \left( U^l(\exp(t\hat{e}_j)) - I_{l\times l} \right) /t \cdot \int_{g\in G} f(g)U^l(g^{-1}) \diff g \\
		&= u^l(\hat{e}_j) \fourier^l[f].
	\end{aligned}
\end{equation}
This enables the conversion of \eqref{eqn:FP_cont} into a set of ODEs of the Fourier coefficients of $p^c(t,g,s)$.

\begin{theorem} \label{thm:FP_cont_ODE}
	Let $\fourier[p^c](t,s)$ be the Fourier coefficients of $p^c(t,g,s)$ which depend on $t$ and $s$.
	For any $s\in S$, if $p^c(t,g,s)$ satisfies \eqref{eqn:FP_cont}, then $\fourier[p^c](t,s)$ approximately satisfies the following ODE:
	\begin{equation} \label{eqn:FP_cont_ODE}
		\frac{\diff}{\diff t} \fourier^l[p^c](t,s) = - \sum_{j=1}^{N_g} u^l(\hat{e}_j) \fourier^l[a_jp^c](t,s) + \sum_{j,k=1}^{N_g} D_{j,k} u^l(\hat{e}_j)u^l(\hat{e}_k) \fourier^l[p^c](t,s).
	\end{equation}
\end{theorem}
\begin{proof}
	Suppose $p^c(t,g,s)$ is approximated by a band-limited \cite{maslen1998efficient} sum of its Fourier series:
	\begin{equation} \label{eqn:idft_G_finite}
		p^c(t,g,s) \approx \sum_{l\in\hat{G}_0} \sum_{m_1,m_2=1}^{d(l)} d(l) \fourier_{m_1,m_2}^l[p^c](t,s) U^l_{m_1,m_2}(g),
	\end{equation}
	where $\hat{G}_0$ is a finite subset of $\hat{G}$.
	Substitute the above equation and a similar band-limited expansion of $\mathcal{L}_c^*p^c$ into \eqref{eqn:FP_cont}, we get for all $s\in S$
	\begin{equation*}
		\sum_{l\in\hat{G}_0} \sum_{m_1,m_2=1}^{d(l)} d(l) \frac{\diff}{\diff t} \fourier_{m_1,m_2}^l[p^c](t,s) U^l_{m_1,m_2}(g) \approx \sum_{l\in\hat{G}_0} \sum_{m_1,m_2=1}^{d(l)} d(l) \fourier_{m_1,m_2}^l[\mathcal{L}_c^*p^c](t,s) U^l_{m_1,m_2}(g).
	\end{equation*}
	Due to the orthogonality of basis, we can equate the Fourier coefficient with the same indices, i.e. $\tfrac{\diff}{\diff t} \fourier^l_{m_1,m_2}[p^c](t,s) = \fourier^l_{m_1,m_2}[\mathcal{L}^*p^c](t,s)$ for any $s\in S$.
	Equation \eqref{eqn:FP_cont_ODE} can be derived by expanding $\mathcal{L}^*p^c$ using the differentiation formula \eqref{eqn:dF_G}.
\end{proof}

Equation \eqref{eqn:FP_cont_ODE} is ODEs of the Fourier coefficients $\fourier_{m_1,m_2}^l[p^c](t,s)$, and can be integrated using numerical ODE solvers.
To illustrate how the calculations can be carried out in practice, we present with the simplest forward Euler's method.
Suppose a grid $\{g_\nu\}_{\nu=1}^{N_q}$ is put onto $G$, and there is a quadrature rule $\{w_\nu\}_{\nu=1}^{N_q}$ such that the integral in \eqref{eqn:dft_G} can be approximated by a finite summation:
\begin{equation} \label{eqn:dft_G_finite}
	\fourier^l[f] = \sum_{\nu=1}^{N_q} w_\nu f(g_\nu) U^l(g_\nu^{-1}).
\end{equation}
This allows the forward Fourier transform to be computed.
Suppose at time $t=t_k$, the values of $p^c(t_k,g_\nu,s)$, $a(t_k,g_\nu,s)$ on the grid, and $D(t_k,s)$ are given.
Then the Fourier coefficients $\fourier^l[a_jp^c](t_k,s)$, and $\fourier^l[p^c](t_k,s)$ for $l\in\hat{G}_0$ can be calculated as in \eqref{eqn:dft_G_finite} using the quadrature rule.
Namely, the right hand side of \eqref{eqn:FP_cont_ODE}, i.e., $\fourier^l[\mathcal{L}_c^*p^c](t_k,s)$ can be calculated, which enables the first order integration $\fourier^l[p^c](t_{k+1},s) = \fourier^l[p^c](t_k,s) + \fourier^l[\mathcal{L}_c^*p^c](t_k,s) \Delta t$.
And the values of $p^c(t_{k+1},g_\nu,s)$ on the grid can be recovered by \eqref{eqn:idft_G_finite}, which can be used in the next integration step.
The pseudocode is summarized in \cref{alog:UP}, where the first order integration can easily be replaced by other higher order numerical integration schemes.

The summations in \eqref{eqn:dft_G_finite} and \eqref{eqn:idft_G_finite} can be accelerated using extensions of the classic Cooley-Tukey FFT algorithm to compact Lie groups.
See for example \cite{rockmore2004recent} for a review on this topic.
The FFT algorithms have been developed for classic compact matrix Lie groups $\SO{n}$, $\mathrm{U}(n)$, $\mathrm{SU}(n)$ and $\mathrm{Sp}(n)$ in \cite{maslen1998efficient}.

\subsection{Propagation over Discrete Dynamics}
Next, consider \eqref{eqn:FP_dist}.
We apply the quadrature rule on $G$ to the integral term on the right hand side of \eqref{eqn:FP_dist}, where the integration over $G$ is replaced by a finite summation over the grid:
\begin{equation} \label{eqn:FP_dist_ODE}
	\frac{\diff p^d(t,g_\nu,s)}{\diff t} = \sum_{s^-\in S} \sum_{\nu'=1}^{N_q} w_{\nu'}\kappa(g_{\nu'}^-,s^-,g_\nu,s) \lambda(g_{\nu'}^-,s^-) p^d(t,g_{\nu'}^-,s^-) - \lambda(g_\nu,s) p^d(t,g_\nu,s).
\end{equation}
The above is a set of linear ODEs of $p^d(t,g_\nu,s)$ with $N_s\cdot N_q$ equations.
Suppose at time $t=t_k$, the values of $p^d(t_k,g_\nu,s)$ on the grid are given, then the values of $p^d$ at time $t=t_{k+1}$ can be integrated using the forward Euler's method as $p(t_{k+1},g_\nu,s) = p(t_k,g_\nu,s) + \mathcal{L}_d^* p^d(t_k,g_\nu,s) \Delta t$, where $\mathcal{L}_d^* p^d(t_k,g_\nu,s)$ is calculated as the right hand side of \eqref{eqn:FP_dist_ODE}.
The pseudocode is summarized in \cref{alog:UP}.

\subsection{Splitting Method}

In summary, \eqref{eqn:FP_cont} is transformed into ODEs of Fourier coefficients, and \eqref{eqn:FP_dist} is converted into ODEs of probability densities on the grid. 
The numerical solutions of these two equations are combined using a first order splitting scheme as follows.
Suppose the time is discretized by a sequence $\{t_k\}_{k=0}^\infty$ with a fixed increment $\Delta t = t_{k+1}-t_k $.
Given $p(t_k,g,s)$, we first solve \eqref{eqn:FP_cont} with the initial condition $p^c(t_k,g,s) = p(t_k,g,s)$  to obtain $p^c(t_{k+1},g,s)$ that is propagated over the continuous dynamics. 
Next, we solve \eqref{eqn:FP_dist} with the initial condition $p^d(t_k,g,s) = p^c(t_{k+1},g,s)$ to propagate it over the discrete dynamics, and construct $p^d(t_{k+1},g,s)$ which is taken as the solution of \eqref{eqn:FP} at $t_{k+1}$.
These two parts are integrated seamlessly, as the probability density values propagated over the discrete dynamics are on the grid designed for the Fourier transform required for the propagation over the continuous dynamics. 
The pseudocode is summarized in \cref{alog:UP}.

\begin{algorithm}
	\caption{Uncertainty propagation for GSHS}
	\label{alog:UP}
	\begin{algorithmic}[1]
		\Procedure{$p(t_{k+1},g_\nu,s) = $ uncertainty\_propagation}{$p(t_k,g_\nu,s)$}
		\State $p^c(t_k,g_\nu,s) = p(t_k,g_\nu,s)$.
		\State $p^c(t_{k+1},g_\nu,s) = $ \textsc{propagate\_continuous}($p^c(t_k,g_\nu,s)$).
		\State $p^d(t_k,g_\nu,s) = p^c(t_{k+1},g_\nu,s)$.
		\State $p^d(t_{k+1},g_\nu,s) = $ \textsc{propagate\_discrete}($p^d(t_k,g_\nu,s)$).
		\State $p(t_{k+1},g_\nu,s) = p^d(t_{k+1},g_\nu,s)$.
		\EndProcedure
		\algrule
		\Procedure{$p^c(t_{k+1},g_\nu,s) = $ propagate\_continuous}{$p^c(t_k,g_\nu,s)$}
		\For{$s\in S$}
		\For{$l\in\hat{G}_0$}
		\For{$j = 1,\ldots,N_g$}
		\State Compute $\fourier^l[a_jp^c](t_k,s)$ using \eqref{eqn:dft_G_finite}.
		\EndFor
		\State Compute $\fourier^l[p^c](t_k,s)$ using \eqref{eqn:dft_G_finite}.
		\State Compute $\fourier^l(\mathcal{L}_c^*p^c)(t_k,s)$ using the right hand side of \eqref{eqn:FP_cont_ODE}.
		\State $\fourier^l[p^c](t_{k+1},s) = \fourier^l[p^c](t_k,s) + \fourier^l(\mathcal{L}_c^*p^c)(t_k,s) \Delta t$.
		\EndFor
		\State Recover $p^c(t_{k+1},g_\nu,s)$ from $\fourier^l[p^c](t_{k+1},s)$ using \eqref{eqn:idft_G_finite}.
		\EndFor
		\EndProcedure
		\algrule
		\Procedure{$p^d(t_{k+1},g_\nu,s) = $ propagate\_discrete}{$p^d(t_k,g_\nu,s)$}
		\For{$s\in S$ and $\nu = 1,\ldots,N_q$}
		\State Compute $\mathcal{L}_d^*p^d(t_k,g_\nu,s)$ using the right hand side of \eqref{eqn:FP_dist_ODE}.
		\State $p^d(t_{k+1},g_\nu,s) = p^d(t_k,g_\nu,s) + \mathcal{L}_d^*p^d(t_k,g_\nu,s)\Delta t$.
		\EndFor
		\EndProcedure
	\end{algorithmic}
\end{algorithm}


\section{Numerical Example of 3D Pendulum} \label{sec:example}

A 3D pendulum is a rigid body that freely rotates about an inertially-fixed pivot under gravity. 
In this section, we apply the proposed method to the 3D pendulum model, to propagate the uncertainties of its attitude and angular velocity.
For the discrete dynamics, we assume that the 3D pendulum may collide with a fixed planar wall, which causes an instantaneous change of its angular velocity (see \cref{fig:pendulum}). 

\subsection{3D Pendulum Model} \label{sec:pendulum}

We use the GSHS defined in Section \ref{sec:GSHS} to model the 3D pendulum as follows. 
Two reference frames are used: the inertial frame $\{\vec e_1,\vec e_2, \vec e_3\}$, and the body-fixed frame of the pendulum $\{\vec b_1, \vec b_2,\vec b_3\}$.
The origin of the body-fixed frame is at the pivot point denoted by $P$.
The continuous state is $(R,\Omega) \in G = \SO{3}\times \real^3$, where $R\in\SO{3}$ is the attitude of the pendulum, i.e., the linear transform of coordinates from the body-fixed frame to the inertial frame, and $\Omega\in\real^3$ is the coordinates of angular velocity in the body-fixed frame.
The discrete state space is $S=\{1\}$, i.e., there is only one discrete mode.
Throughout this section, for any vector $\vec{a}$, $a\in\real^3$ denotes its coordinates in the inertial frame, if not stated otherwise.

\paragraph{Continuous Dynamics}

The continuous dynamics is given by the following SDE:
\begin{subequations} \label{eqn:pendulum_cont}
	\begin{align}
		R^T\diff R &= \hat\Omega \diff t, \\
		\diff\Omega &= \left( J^{-1}(-\Omega \times J\Omega - mg\rho\times R^Te_3) - B\Omega \right) \diff t + H_c\diff W_t,
	\end{align}
\end{subequations}
where $J\in\real^{3\times 3}$ is the moment of inertia about the pivot, and $m\in\real$ is the mass.
The coordinates of the center of mass are given by $\rho\in\real^3$ in the body-fixed frame.
The fixed gravitational acceleration is denoted by $g\in\real$.
There is a damping torque proportional to the angular velocity scaled by the matrix $B = \diag(B_1,B_2,B_3) \in \real^{3\times 3}$.
Finally, $W_t\in\real^3$ is the standard Wiener process, representing random external torques. 

\begin{figure}
	\centering
	\begin{subfigure}{0.49\textwidth}
		\centering
		\includegraphics{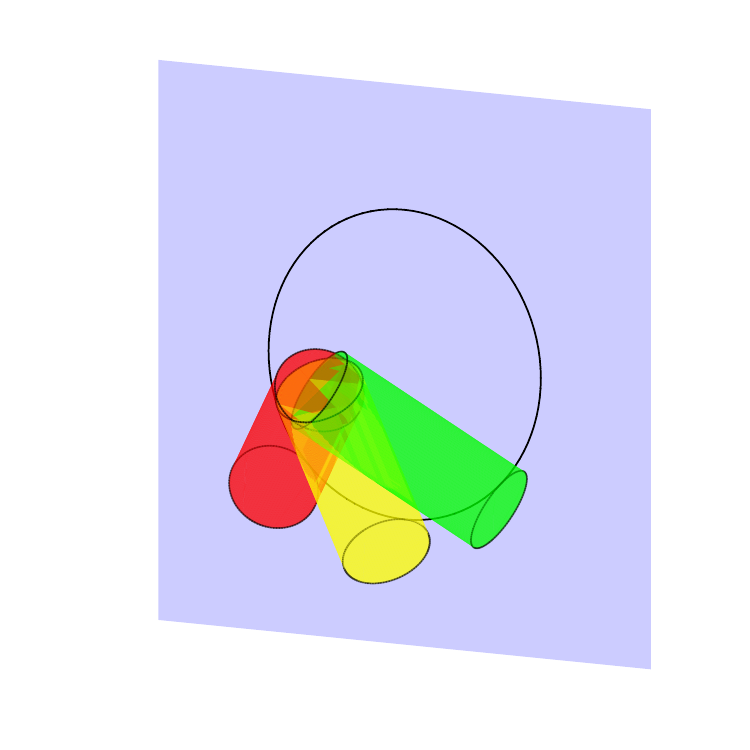}
		\caption{Possible collision configurations. \label{fig:pendulum-collisonConf}}
	\end{subfigure}
	\begin{subfigure}{0.49\textwidth}
		\includegraphics{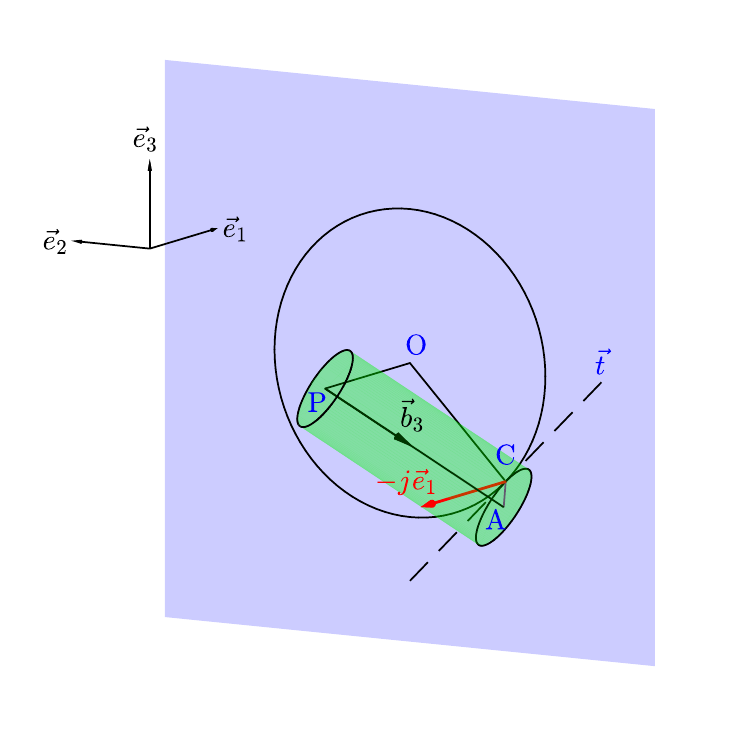}
		\caption{Collision response. \label{fig:pendulum-collisionResp}}
	\end{subfigure}
	\caption{An Axially symmetric 3D pendulum colliding with a wall. \label{fig:pendulum}}
\end{figure}

We make the following assumptions to simplify the continuous dynamics: (i) the pendulum is axially symmetric, i.e., $J = \diag(J_1,J_1,J_3)$, and $\rho = [0,0,\rho_z]^T$ is along the axis $\vec b_3$;
(ii) $Init(\{\Omega_3 = 0\}) = 1$, i.e., the initial angular velocity along the axis of symmetry is zero with probability one;
(iii) the third row of $H_c$ is zero.
Under these assumptions, it is straightforward to verify that $\mathbb{P}(\Omega_3(t)=0) = 1$ for any $t>t_0$.
As a consequence, we may ignore $\Omega_3$ and reduce the continuous state space into $G = \SO{3}\times \real^2$.
The resulting SDE is given by
\begin{subequations} \label{eqn:pendulum_cont_5D}
	\begin{align}
		R^T\diff R &= \left( \begin{bmatrix} \tilde{\Omega}^T & 0 \end{bmatrix}^T \right)^\wedge \diff t, \\
		\diff \tilde{\Omega} &= \left( \frac{mg\rho_z}{J_1} \begin{bmatrix} R_{32} \\ -R_{31} \end{bmatrix} - \tilde{B}\tilde{\Omega} \right) \diff t + \tilde{H}_c\diff W_t,
	\end{align}
\end{subequations}
where $\tilde{\Omega} = [\Omega_1, \Omega_2]^T$, $\tilde{B} = \diag(B_1,B_2)$, and $\tilde{H}_c$ is the first two rows of $H_c$.

\paragraph{Discrete Dynamics}

A planar wall is placed perpendicular to the inertial $\vec e_1$ axis, at $d_{\text{wall}}>0$ from the pivot point $P$.
As the pendulum swings, it may collide with the wall and rebound.
We further assume the pendulum is a cylinder with the height $h$ and the radius $r$.
Then, all possible collision points between the pendulum and the wall form a circle, as illustrated in \cref{fig:pendulum-collisonConf}.
Let the angle between the $\vec{e}_1$-$\vec{e}_2$ plane and $\vec b_3$ be denoted by $\theta = \arcsin(\vec{b}_3\cdot\vec{e}_1)$.
A collision occurs when
\begin{subequations} \label{eqn:collsion_cond}
	\begin{gather}
		\theta \geq \theta_0 = \arcsin\frac{d_{\text{wall}}}{\sqrt{h^2+r^2}} - \arcsin\frac{r}{\sqrt{h^2+r^2}}, \\
		(R\Omega \times \varrho) \cdot e_1 > 0,
	\end{gather}
\end{subequations}
where $\varrho\in\real^3$ is the coordinates of the vector $\overrightarrow{PC} = (h-r\tan\theta) \vec b_3 + r\sec\theta \vec e_1$ in the inertial frame, and $C$ is the point on the pendulum that has the largest coordinate along $\vec e_1$.

The first equation states that the pendulum penetrates through the wall, and the second equation implies that the pendulum is rotating towards the wall.
Equation \eqref{eqn:collsion_cond} represents a guard set defined such that whenever the continuous state enters it, the discrete jump is triggered.
This corresponds to a deterministic forced jump of hybrid systems.
In the presented GSHS, a Poisson process can be designed to approximate this forced jump, with a rate function $\lambda(R,\Omega)$ being very large when \eqref{eqn:collsion_cond} is satisfied, and zero otherwise.
However, one must make a compromise on the space variation of $\lambda$.
Specifically, the ideal $\lambda$ would make the probability density $p^c$ large outside the guard set, and close to zero inside, i.e., there is a large space variation of $p^c$ caused by the discontinuity at the boundary.
This is unfavourable for the spectral method, since a high bandwidth must be used to capture the large space variation, at the cost of increased computational load.
Here we design a rate function corresponding to \eqref{eqn:collsion_cond} with a relatively small space variation as
\begin{equation} \label{eqn:pendulum_lambda}
	\lambda(R,\Omega) = \begin{cases}
		\frac{\lambda_{\text{max}}}{2} \sin\left( \frac{pi}{2\theta_t}(\theta-\theta_0) \right) + \frac{\lambda_{\text{max}}}{2}, & \text{if} \ -\theta_t \leq \theta - \theta_0 \leq \theta_t, \ R\Omega \times \varrho \cdot e_1 > 0 \\
		\lambda_{\text{max}}, & \text{if} \ \theta - \theta_0 > \theta_t, \ R\Omega \times \varrho \cdot e_1 > 0 \\
		0, & \text{otherwise}
	\end{cases}.
\end{equation}
Namely, a threshold $\theta_t > 0$ is used to mark a ``boundary region'' of the guard set, and the sine function is used to make a smooth connection between the large $\lambda_{\text{max}} > 0$ inside the guard set, and zero outside (\cref{fig:lambda}).

\begin{figure}
	\centering
	\includegraphics{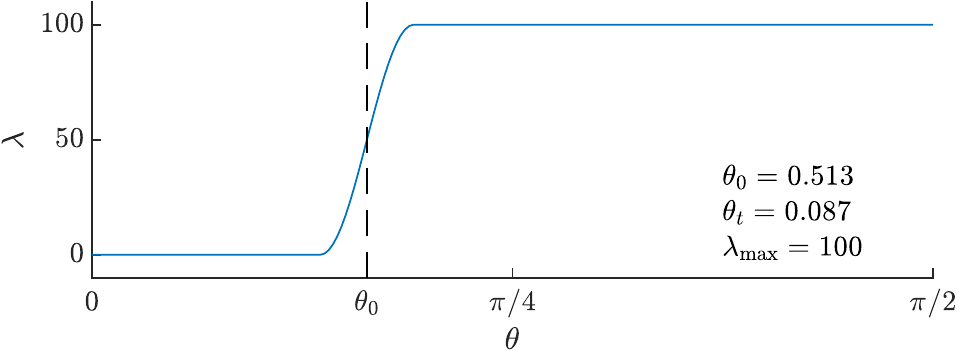}
	\caption{$\lambda(R,\Omega)$ versus $\theta(R)$ when $(R\Omega \times \varrho)\cdot e_1 > 0$. \label{fig:lambda}}
\end{figure}

Next, we formulate the stochastic kernel describing the state distribution immediately after a jump. 
During the collision, an impulse $-j\vec e_i$ with $j>0$ is applied to the pendulum at the collision point $C$ (\cref{fig:pendulum-collisionResp}) that redirects the linear velocity of the pendulum at $C$ along $\vec e_1$.
First assume there is no noise, then the collision response can be summarized in terms of the change of angular velocity and linear velocity at $C$, as follows:
\begin{subequations}
	\begin{gather}
		\bar{\Omega}^+ - \Omega^- = J^{-1}R^T\left( \varrho \times (-je_1) \right), \\
		(R\bar{\Omega}^+ \times \varrho) \cdot e_1 = -\varepsilon (R\Omega^- \times \varrho) \cdot e_1, \label{eqn:collisionResp_linear}
	\end{gather}
\end{subequations}
where $\Omega^-$, $\bar{\Omega}^+$ denote the angular velocities before and after collision respectively, and $0<\varepsilon\leq 1$ is the coefficient of restitution.
Note that $t \triangleq (\varrho\times e_1) / |\varrho\times e_1|$ is perpendicular to $b_3$, since $J = \diag(J_1,J_1,J_3)$, it can be verified that $J^{-1}R^Tt = \tfrac{1}{J_1}R^Tt$.
This indicates that $\bar{\Omega}^+-\Omega^-$ is along $R^Tt$ and is perpendicular to $R^Tb_3$, thus $\bar{\Omega}_3^+ = \Omega_3^-$ and we may ignore $\Omega_3$ as we did in the continuous dynamics where $\mathbb{P}(\Omega_3=0) = 1$.
Furthermore, \eqref{eqn:collisionResp_linear} can be simplified into
\begin{equation}
	\bar{\Omega}^+ = \Omega^- - (1+\varepsilon) (\Omega^- \cdot R^Tt) R^Tt,
\end{equation}
which gives the continuous state $\Omega$ right after the collision in an ideal case. 
Here we further assume that the angular velocity is also perturbed by a Gaussian noise during collision, i.e.,
\begin{equation}
	\Omega^+ = \bar{\Omega}^+ + \begin{bmatrix} H_d\xi \\ 0 \end{bmatrix},
\end{equation}
where $\Omega^+$ is the perturbed angular velocity after collision, $H_d\in\real^{2\times 2}$, and $\xi$ is a 2-dimensional standard Gaussian random vector.

In short, the stochastic kernel for discrete jump caused by the collision can be written as
\begin{equation} \label{eqn:pendulum_kappa}
	\begin{aligned}
		\kappa(R^-,\tilde{\Omega}^-,R^+,\tilde{\Omega}^+) = &\delta_{\SO{3}}\left( R^+(R^-)^T \right) \\
		&\quad \times \frac{1}{2\pi\sqrt{\det \Sigma_d}} \exp\left\{ -\tfrac{1}{2} \left(\tilde{\Omega}^+ - \tilde{\bar{\Omega}}_0^+\right)^T \Sigma_d^{-1} \left(\tilde{\Omega}^+ - \tilde{\bar{\Omega}}_0^+\right) \right\},
	\end{aligned}
\end{equation}
where $\Sigma_d = H_dH_d^T$, and $\delta_{\SO{3}}$ is the Dirac-delta function on $\SO{3}$.
In other words, the discrete jump does not alter the attitude $R$, while it resets the angular velocity from $\tilde{\Omega}^-$ to $\tilde{\bar{\Omega}}^+$ with Gaussian distributed random noises.

\subsection{Harmonic Analysis on $\SO3\times\torus^2$} \label{sec:pendulum_dft}

One obstacle in applying the proposed approach to the pendulum model is that the continuous state space, namely $\SO3\times \real^2$ is not compact.
Nevertheless, since the angular velocity is uniformly bounded by the initial mechanical energy of the pendulum, as long as $p(t_0,R,\tilde{\Omega})$ is compactly supported, we may assume $p(t,R,\tilde{\Omega})$ is compactly supported, uniformly in time $t$.
Therefore, the continuous state space can be regarded as $\SO{3}\times \torus^2$, where $\torus^2 = \sph^1 \times \sph^1$ is the 2-dimensional torus.
Noncommutative harmonic analysis on $\SO3$ has been presented in~\cite{chirikjian2000engineering,varshalovich1988quantum}, and harmonic analysis on $\sph$ is widely available. 
Here we review those materials needed to formulate harmonic analysis on $\SO3\times\torus^2$.

\paragraph{Representation}

Let the representations of $\SO3$ be denoted by $\{U^l(R)\}_{l\in\mathbb{N}}$, where the dimension of $U^l(R)$ is $d(l) = 2l+1$.
Suppose that $R\in\SO3$ is parameterized by the 3-2-3 Euler angles as
\begin{equation}
	R(\alpha,\beta,\gamma) = \exp(\alpha \hat e_3) \exp(\beta \hat e_2) \exp(\gamma \hat e_3),
\end{equation}
where $\alpha,\gamma\in[0,2\pi)$, $\beta\in[0,\pi]$.
For $-l \leq m_1, m_2\leq l$, the elements of the $l$-th representation $U^l(R)$ can be explicitly written as 
\begin{equation}
    U_{m_1,m_2}^l(R(\alpha,\beta,\gamma)) = e^{-im_1\alpha} d_{m_1,m_2}^l(\beta) e^{-im_2\gamma},
\end{equation}
where $d_{m_1,m_2}^l(\beta)$ is the real valued Wigner-d function~\cite{varshalovich1988quantum}.
Next, the representations of $\torus^2$ are given by
\begin{equation}
	V^n(\tilde{\Omega}) = \exp\left( \frac{i\pi n_1\Omega_1}{L} + \frac{i\pi n_2\Omega_2}{L} \right)
\end{equation}
for $n=(n_1,n_2)\in\mathbb{Z}^2$,
where $(\Omega_1,\Omega_2)$ is normalized by its uniform bound $L>0$, such that $\pi\Omega_j/L \in [-\pi,\pi)$, $j=1,2$, so $\tilde{\Omega} \in \torus^2$.
Then, the representations of $\SO3\times\torus^2$ are given by the tensor product of $U^l$ and $V^n$, and more explicitly $\{\{U_{m_1,m_2}^l(R) \cdot V^n(\tilde{\Omega})\}_{m_1,m_2=-l}^{l}\}_{l\in\mathbb{N}, n\in\mathbb{Z}^2}$,
which forms a complete orthonormal basis for the function space $L^2(\SO{3}\times \torus^2)$.

\paragraph{Sampling Theorem}

Consider a band-limited function on $\SO3\times \torus^2$ spanned by the representations with $l \leq l_0-1$ and $-n_0 \leq n_1, n_2 \leq n_0-1$.
According to the sampling theorem, its Fourier coefficients can be exactly recovered by the sample values on a certain grid and the associated quadrature rule. 
The grid on $\SO{3}$ can be designed in terms of Euler angles, with $2l_0$ points along each dimension:
\begin{equation}
	\alpha_{\nu_1} = \frac{\pi\nu_1}{l_0}, \qquad \beta_{\nu_2} = \frac{\pi(2\nu_2+1)}{4l_0}, \qquad \gamma_{\nu_3} = \frac{\pi\nu_3}{l_0},
\end{equation}
for $\nu_1, \nu_2, \nu_3 = 0,\ldots,2l_0-1$.
The quadrature rule associated with this grid \cite{kostelec2008ffts} is
\begin{equation}
	w_\nu = \frac{1}{4l_0^3} \sin(\beta_{\nu_2}) \sum_{j=0}^{l_0-1} \frac{1}{2j+1} \sin((2j+1)\beta_{\nu_2}).
\end{equation}
Similarly, a grid on $\torus^2$ can be defined as
\begin{equation} \label{eqn:grid_T2}
	\frac{\pi(\Omega_1)_{\mu_1}}{L} = \frac{\mu_1\pi}{n_0}, \qquad \frac{\pi(\Omega_2)_{\mu_2}}{L} = \frac{\mu_2\pi}{n_0},
\end{equation}
for $\mu_1, \mu_2 = -n_0,\ldots,n_0-1$.
The quadrature rule is simply
\begin{equation}
	w_{\mu} = \frac{1}{(2n_0)^2}.
\end{equation}

Using the above orthonormal basis and quadrature rules, the forward Fourier transform \eqref{eqn:dft_G} can be computed using the following finite summation:
\begin{equation} \label{eqn:dft_SO3_R2}
	\begin{aligned}
		\fourier_{m_1,m_2}^{l,n}[f] &= \sum_{\mu_1, \mu_2 = -n_0}^{n_0-1} \sum_{\nu_1,\nu_2,\nu_3=0}^{2l_0-1} w_{\nu} w_{\mu} f\Big( R(\alpha_{\nu_1},\beta_{\nu_2},\gamma_{\nu_3}), (\Omega_1)_{\mu_1}, (\Omega_2)_{\mu_2} \Big) \\
		&\quad \cdot e^{im_1\alpha_{\nu_1}} d_{m_1,m_2}^l(\beta_{\nu_2}) e^{im_2\gamma_{\nu_3}} \cdot \exp\left( -\frac{in_1\pi(\Omega_1)_{\mu_1}}{L} - \frac{in_2\pi(\Omega_2)_{\mu_2}}{L} \right),
	\end{aligned}
\end{equation}
for any $0\leq l\leq l_0-1$, $-l\leq m_1, m_2 \leq l$, and $-n_0\leq n_1, n_2 \leq n_0-1$.
Conversely, the backward Fourier transform \eqref{eqn:idft_G_finite} can be explicitly written as
\begin{equation} \label{eqn:idft_SO3_R2}
	\begin{aligned}
		&f\Big( R(\alpha_{\nu_1},\beta_{\nu_2},\gamma_{\nu_3}), (\Omega_1)_{\mu_1}, (\Omega_2)_{\mu_2} \Big) = \sum_{n_1,n_2 = -n_0}^{n_0-1} \sum_{l=0}^{l_0-1} \sum_{m_1,m_2 = -l}^l (2l+1) \fourier_{m_1,m_2}^{l,n}[f] \\
		&\qquad \qquad \cdot e^{-im_1\alpha_{\nu_1}} d_{m_1,m_2}^l(\beta_{\nu_2}) e^{-im_2\gamma_{\nu_3}} \cdot \exp\left( \frac{in_1\pi(\Omega_1)_{\mu_1}}{L} + \frac{in_2\pi(\Omega_2)_{\mu_2}}{L} \right),
	\end{aligned}
\end{equation}
to recover the function values on the grid.
The summations in \eqref{eqn:dft_SO3_R2} and \eqref{eqn:idft_SO3_R2} can be computed using a combination of the classic Cooley-Tukey FFT algorithm, and the FFT developed specifically for $\SO{3}$ \cite{kostelec2008ffts}.

\subsection{Implementation} \label{sec:pendulum_algo}

Now, the proposed method is implemented to the pendulum model as follows.

\paragraph{Continuous Dynamics}

First, for the continuous dynamics \eqref{eqn:pendulum_cont_5D}, the corresponding FP equation \eqref{eqn:FP_cont} can be written as
\begin{equation} \label{eqn:pendulum_cont_FP}
	\begin{aligned}
		\frac{\partial p^c(t,R,\tilde{\Omega})}{\partial t} = &-\sum_{j=1}^2 \liediff_j \big( \Omega_j p^c(t,R,\tilde{\Omega}) \big) - \sum_{j=1}^2  \frac{\partial}{\partial \Omega_j} \big( a_j^g(R)p^c(t,R,\tilde{\Omega}) \big) \\
		&+ \sum_{j=1}^2 B_j \frac{\partial}{\partial \Omega_j} \big( \Omega_j p^c(t,R,\tilde{\Omega}) \big) + \sum_{j,k=1}^2 \tilde{D}_{jk} \frac{\partial}{\partial \Omega_i \partial \Omega_j} p^c(t,R,\tilde{\Omega}),
	\end{aligned}
\end{equation}
where $a^g(R) = \tfrac{mg\rho_z}{J_1}[R_{32}, -R_{31}]^T$, and $\tilde{D} = \tfrac{1}{2}\tilde{H}_c\tilde{H}_c^T$.
Next, we present selected operational properties of the representations that are required to perform the Fourier transform for the right hand side of the above expression. 
For the representation $U^l(R)$ of $\SO{3}$, the associated Lie algebra representation has explicit forms:
\begin{subequations}
	\begin{align}
		u_{m_1,m_2}^l(\hat e_1) &= -\tfrac{1}{2}ic_{m_2}^l\delta_{m_1-1,m_2} - \tfrac{1}{2}ic_{-m_2}^l\delta_{m_1+1,m_2}, \\
		u_{m_1,m_2}^l(\hat e_2) &= -\tfrac{1}{2}c_{m_2}^l\delta_{m_1-1,m_2} + \tfrac{1}{2}c_{-m_2}^l\delta_{m_1+1,m_2}, \\
		u_{m_1,m_2}^l(\hat e_3) &= -im_1\delta_{m_1,m_2},
	\end{align}
\end{subequations}
where $c_m^l = \sqrt{(l-m)(l+m+1)}$.
These can be used to calculate $\fourier^{l,n}[\liediff_j(\Omega_jp^c)]$, the first term on the right hand side of \eqref{eqn:pendulum_cont_FP}, as in \eqref{eqn:dF_G}.
Also, the Lie algebra representation associated with $V^n(\tilde{\Omega})$ is
\begin{equation}
	v^n(e_j) = \frac{i\pi n_j}{L}, \qquad j = 1,2,
\end{equation}
which can be used to obtain those terms involving $\fourier^{l,n}[\tfrac{\partial }{\partial \Omega_j}]$ and $\fourier^{l,n}[\tfrac{\partial^2 p^c}{\partial \Omega_j \partial \Omega_k}]$ in \eqref{eqn:pendulum_cont_FP}.
Furthermore, utilizing the inverse convolution theorem specific to the Fourier series on $\torus^N$, for any $f\in L^2(\SO{3}\times \torus^2)$ and $g\in L^2(\torus^2)$, we may calculate the Fourier coefficient of $f\cdot g$ directly from those of $f$ and $g$:
\begin{equation} \label{eqn:dft-product}
	\mathcal{F}_{m_1,m_2}^{l,n}[fg] = \sum_{n_1',n_2' = -n_0}^{n_0-1} \fourier_{m_1,m_2}^{l,n-n'}[f] \fourier^{n'}[g],
\end{equation}
where $n-n' = (n_1-n_1',n_2-n_2')$.
This can be used for $\fourier^{l,n}[\Omega_jp^c]$ as an intermediate step to calculate $\fourier^{l,n}[\liediff_j(\Omega_jp^c)]$.
Using these properties, the Fourier transform of the right hand side for the continuous dynamics \eqref{eqn:pendulum_cont_FP} can be constructed.

\paragraph{Discrete Dynamics}

Next, we consider the discrete dynamics of the pendulum given by \eqref{eqn:pendulum_lambda} and \eqref{eqn:pendulum_kappa}.
The FP equation \eqref{eqn:FP_dist} can be simplified as
\begin{align}
	&\frac{\partial p^d(t,R,\tilde{\Omega})}{\partial t} =  \int_{\tfrac{\pi\tilde{\Omega}^-}{L} \in \torus^2} \kappa_{\Omega}(R,\tilde{\Omega}^-,\tilde{\Omega}) \lambda(R,\tilde{\Omega}^-) p^d(t,R,\tilde{\Omega}^-) \diff\tilde{\Omega}^- - \lambda(R,\tilde{\Omega})p^d(t,R,\tilde{\Omega}) \label{eqn:pendulum_dist_FP} 
\end{align}
where $\kappa_{\Omega}(R^-,\Omega^-,\Omega^+)$ is the second term of the right hand side of \eqref{eqn:pendulum_kappa}, and $\diff\Omega^-$ is the Lebesgue measure on $\real^2$.
As the attitude remains unchanged after any collision, the domain of integration in the above expression has been reduced to $\torus^2$.
More specifically, we have used the property of Dirac-delta function: for any continuous $f$ and $R_0\in\SO{3}$, $\int_{R\in\SO{3}} \delta(RR_0^T) f(R) \diff R = f(R_0)$.

For numerical implementation, the integral in \eqref{eqn:pendulum_dist_FP} can be evaluated with a finite summation using the grid \eqref{eqn:grid_T2} on $\torus^2$ as in \eqref{eqn:pendulum_dist_ODE}, where the quadrature weights corresponding to $\diff\tilde{\Omega}^-$ is $w'_\mu = \tfrac{L^2}{(2n_0)^2}$.
\begin{equation} \label{eqn:pendulum_dist_ODE}
	\begin{aligned}
		\frac{\partial p^d(t,R_\nu,\tilde{\Omega}_{\mu})}{\partial t} 
		&\approx \sum_{\mu'_1, \mu'_2 = -n_0}^{n_0-1} w'_{\mu'} \kappa_{\Omega}(R_\nu,\tilde{\Omega}_{\mu'}^-,\tilde{\Omega}_\mu) \lambda(R_\nu,\tilde{\Omega}_{\mu'}^-) p^d(t,R_\nu,\tilde{\Omega}_{\mu'}^-) \\
		&\qquad - \lambda(R_\nu,\tilde{\Omega}_\mu) p^d(t,R_\nu,\tilde{\Omega}_\mu).
	\end{aligned} 
\end{equation}
The density values propagated over the discrete dynamics on the grid can be directly utilized in the subsequent propagation over the continuous dynamics according to the splitting method. 

\subsection{Simulation Results}

The attitude and angular velocity uncertainties of the pendulum model are propagated using \cref{alog:UP}, with the explicit computations developed in Section \ref{sec:pendulum_dft} and Section \ref{sec:pendulum_algo}.
These are implemented in \texttt{c} with Nvidia GPU computing toolkit 11.2, along with a Matlab interface.
The standard FFT on $\torus^N$ is computed using \texttt{cuFFT}, and the discrete convolution in \eqref{eqn:dft-product} and the finite summation in \eqref{eqn:pendulum_dist_ODE} are computed using \texttt{cuTENSOR} 1.2.2.
All of the computations are in double precision.
The code is available at \cite{code}.
For $l_0=n_0=30$, the computation time of propagating over one time step is 245 seconds in average with a Nvidia A100-40GB GPU.

The initial uncertainty distribution is chosen as follows. 
The initial attitude follows a matrix Fisher distribution \cite{mardia2009directional} with parameter $F = \exp(-\frac{2\pi}{3}\hat e_2) \diag(15,15,15)$, i.e., the mean attitude is $\exp(-\frac{2\pi}{3}\hat e_2)$, and the variance is approximately $\tfrac{1}{30}$\SI{}{\radian^2} along each axis (\cref{fig:R_cont}(a)).
The initial angular velocity is Gaussian with zero mean and the standard deviation \SI{2}{\radian\per\second} (\cref{fig:Omega_cont}(a)).
The initial attitude and angular velocity are independent.
The parameters for the pendulum model and those designed for the computation are listed in \cref{tab:parameters}.
Equation \eqref{eqn:FP_cont} is integrated using the fourth order Runge-Kutta method, and \eqref{eqn:FP_dist} is propagated by the forward Euler's method.
The simulation is carried out for eight seconds with the step size of $\Delta t = \SI{0.0025}{\second}$.

\begin{table}
	\centering
	\scriptsize
	\caption{Simulation Parameters \label{tab:parameters}}
	\begin{tabularx}{0.8\textwidth}{cX}
		\toprule
		Dimensions & \hfill \makecell{$h$ \\ \SI{0.2}{\meter}}
					 \hfill \makecell{$r$ \\ \SI{0.025}{\meter}}
					 \hfill \makecell{$\rho_z$ \\ \SI{0.1}{\meter}}
					 \hfill \makecell{$d_{\text{wall}}$ \\ \SI{0.12}{\meter}}
					 \hfill \null \\
		\midrule
		Inertia & \hfill \makecell{$m$ \\ \SI{1.0642}{\kilogram}}
		          \hfill \makecell{$J_1$ \\ \SI{0.0144}{\kilogram\meter\squared}}
		          \hfill \makecell{$g$ \\ \SI{9.8}{\meter\per\second\squared}}
		          \hfill \null \\
		\midrule
		Cont. Dynamics & \hfill \makecell{$\tilde{B}$ \\ $\diag(0.2,0.2)$ [\SI{}{\per\second}]}
		                 \hfill \makecell{$\tilde{H}_c$ \\ $\begin{bmatrix} 1 & 0 & 0 \\ 0 & 1 & 0 \end{bmatrix}$ [\SI{}{\radian\second^{-3/2}}]}
		                 \hfill \null \\
		\midrule
		Dist. Dynamics & \hfill \makecell{$\theta_t$ \\ \SI{5}{\deg}}
		                 \hfill \makecell{$\lambda_{\text{max}}$ \\ 100}
		                 \hfill \makecell{$\varepsilon$ \\ 0.8}
		                 \hfill \makecell{$H_d$ \\ $\diag(0.05,0.05)$ [\SI{}{\radian\per\second}]}
		                 \hfill \null \\
		\midrule
		Computation & \hfill \makecell{$l_0$, $n_0$ \\ 30}
		              \hfill \makecell{$L$ \\ \SI{14.5}{\radian\per\second}}
		              \hfill \makecell{$\Delta t$ \\ \SI{0.0025}{\second}}
		              \hfill \null \\
		\bottomrule
	\end{tabularx}
\end{table}

\paragraph{Propagation of Continuous Dynamics}
We first propagate the uncertainty of the pendulum without collisions, i.e., only \eqref{eqn:FP_cont} is integrated.
The attitude uncertainty is depicted by the marginal distribution of the coordinates of body-fixed base axes via color shading in \cref{fig:R_cont}.
In other words, the red, greed, and blue shades represent the marginal density of the $\vec b_1$, $\vec b_2$, and $\vec b_3$ axes respectively, and darker color indicates larger density value.
Initially, the attitude is concentrated where $\vec b_3$ is about \SI{60}{\deg} to the vertical position (\cref{fig:R_cont}(a)).
After the pendulum is released, it swings about the $\vec e_2$ axis until it reaches the opposite limit position (\cref{fig:R_cont}(d)), and swings back to somewhere slightly below the initial position (\cref{fig:R_cont}(g)), which is repeated later on.
At the same time, the uncertainty spreads about the $\vec b_3$ axis, i.e., the rotation about $\vec b_3$ axis becomes more and more dispersed.
After several cycles, the $\vec b_3$ axis becomes concentrated near the vertical direction (\cref{fig:R_cont}(q)-(t)), since the energy is dissipated by the damping of the pendulum.
Also, the rotation about $\vec b_3$ finally becomes almost uniformly distributed.
The same process can be observed from the bottom view for the distribution of $\vec b_3$ axis shown in \cref{fig:b3_cont}.

\begin{figure}
	\centering
	\begin{subfigure}{0.49\textwidth}
		\centering
		\includegraphics{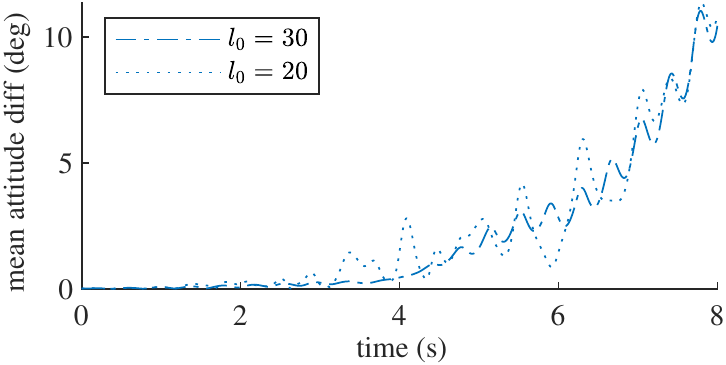}
		\caption{Difference of mean attitude \label{subfig:R_diff}}
	\end{subfigure}
	\begin{subfigure}{0.49\textwidth}
		\centering
		\includegraphics{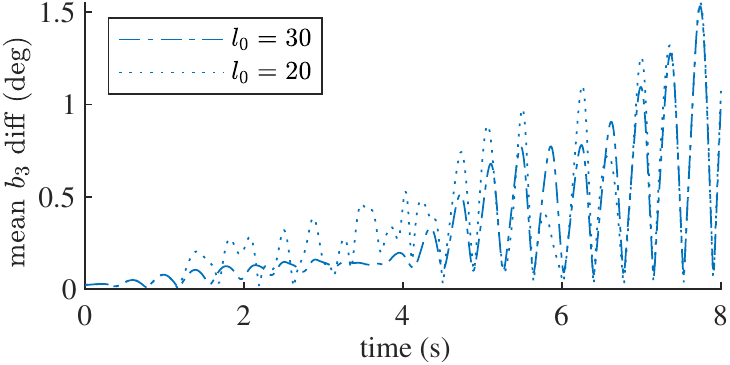}
		\caption{Difference of the mean of $b_3$ \label{subfig:b3_diff}}
	\end{subfigure}
	\begin{subfigure}{0.49\textwidth}
		\centering
		\includegraphics{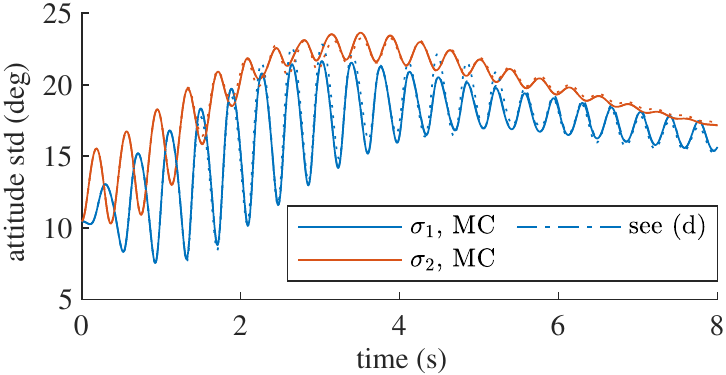}
		\caption{Attitude dispersion \label{subfig:std_R}}
	\end{subfigure}
	\begin{subfigure}{0.49\textwidth}
		\centering
		\includegraphics{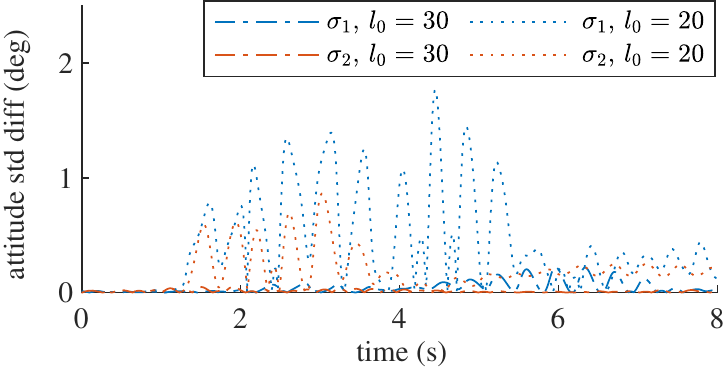}
		\caption{Difference of attitude dispersion \label{subfig:std_Rdiff}}
	\end{subfigure}
	\begin{subfigure}{0.49\textwidth}
		\centering
		\includegraphics{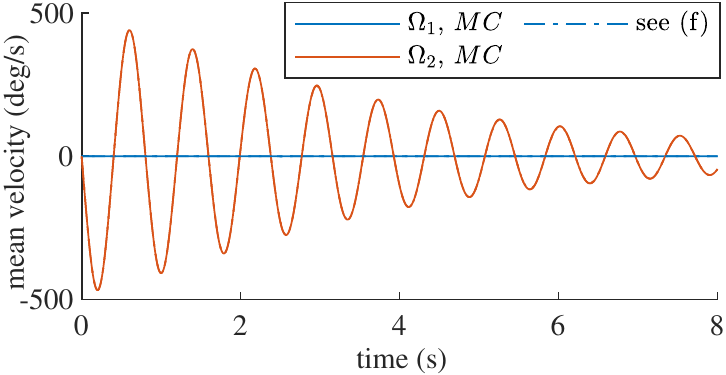}
		\caption{Mean angular velocity \label{subfig:mean_Omega}}
	\end{subfigure}
	\begin{subfigure}{0.49\textwidth}
		\centering
		\includegraphics{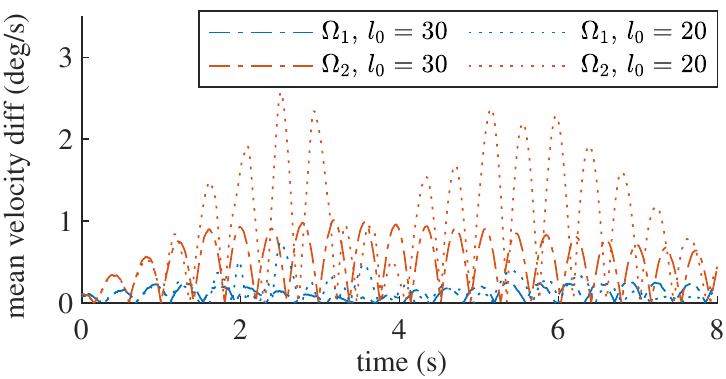}
		\caption{Difference of mean angular velocity \label{subfig:mean_Omegadiff}}
	\end{subfigure}
	\begin{subfigure}{0.49\textwidth}
		\centering
		\includegraphics{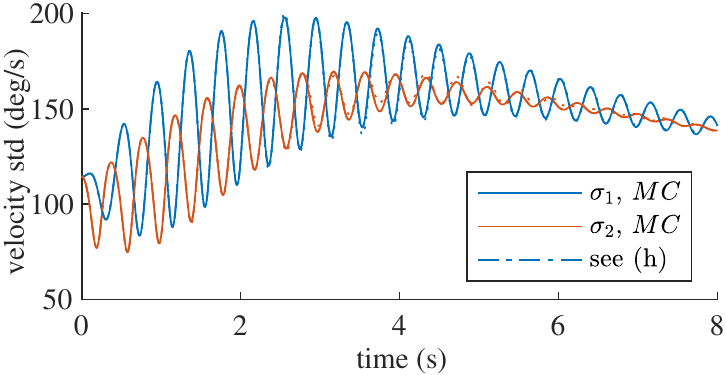}
		\caption{Angular velocity dispersion \label{subfig:std_Omega}}
	\end{subfigure}
	\begin{subfigure}{0.49\textwidth}
		\centering
		\includegraphics{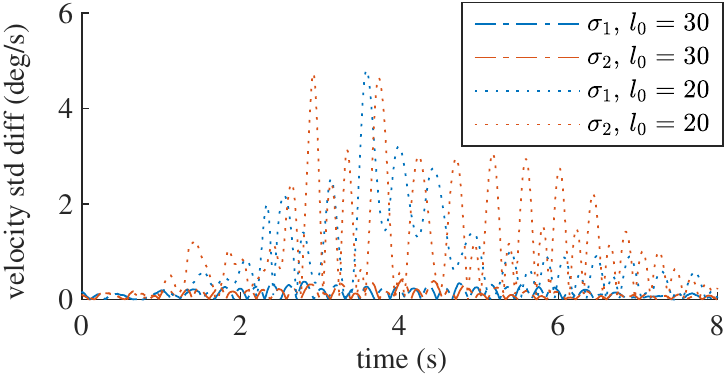}
		\caption{Difference of angular velocity dispersion \label{subfig:std_Omegadiff}}
	\end{subfigure}
	\caption{Comparison of proposed method with Monte Carlo simulation without collisions. \label{fig:compare}}
\end{figure}

\begin{figure}
	\centering
	\includegraphics{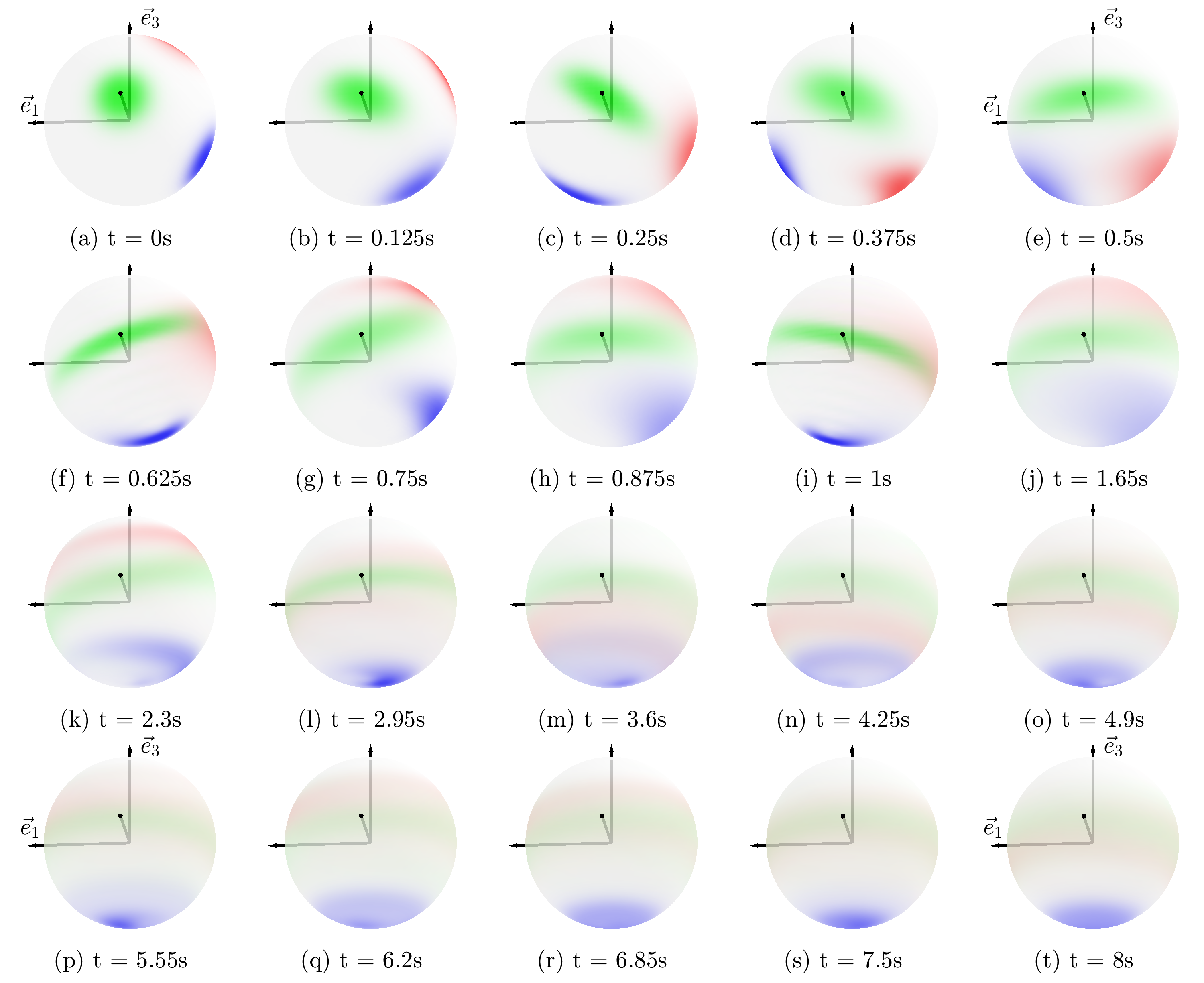}
	\caption{Marginal distribution of attitude without collisions. \label{fig:R_cont}}
\end{figure}

\begin{figure}
	\centering
	\includegraphics{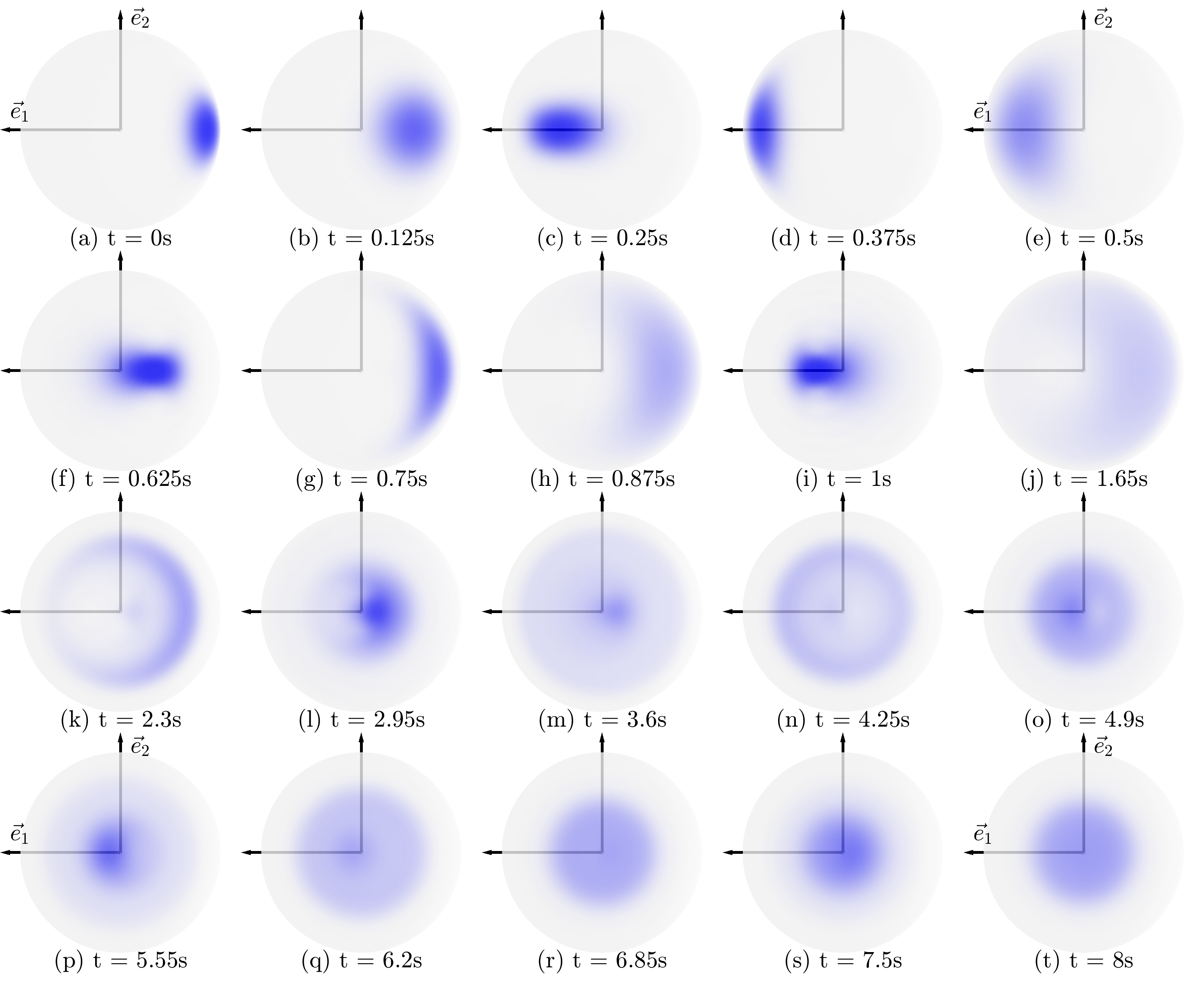}
	\caption{Marginal distribution of $b_3$ without collisions. \label{fig:b3_cont}}
\end{figure}

\begin{figure}
	\centering
	\includegraphics{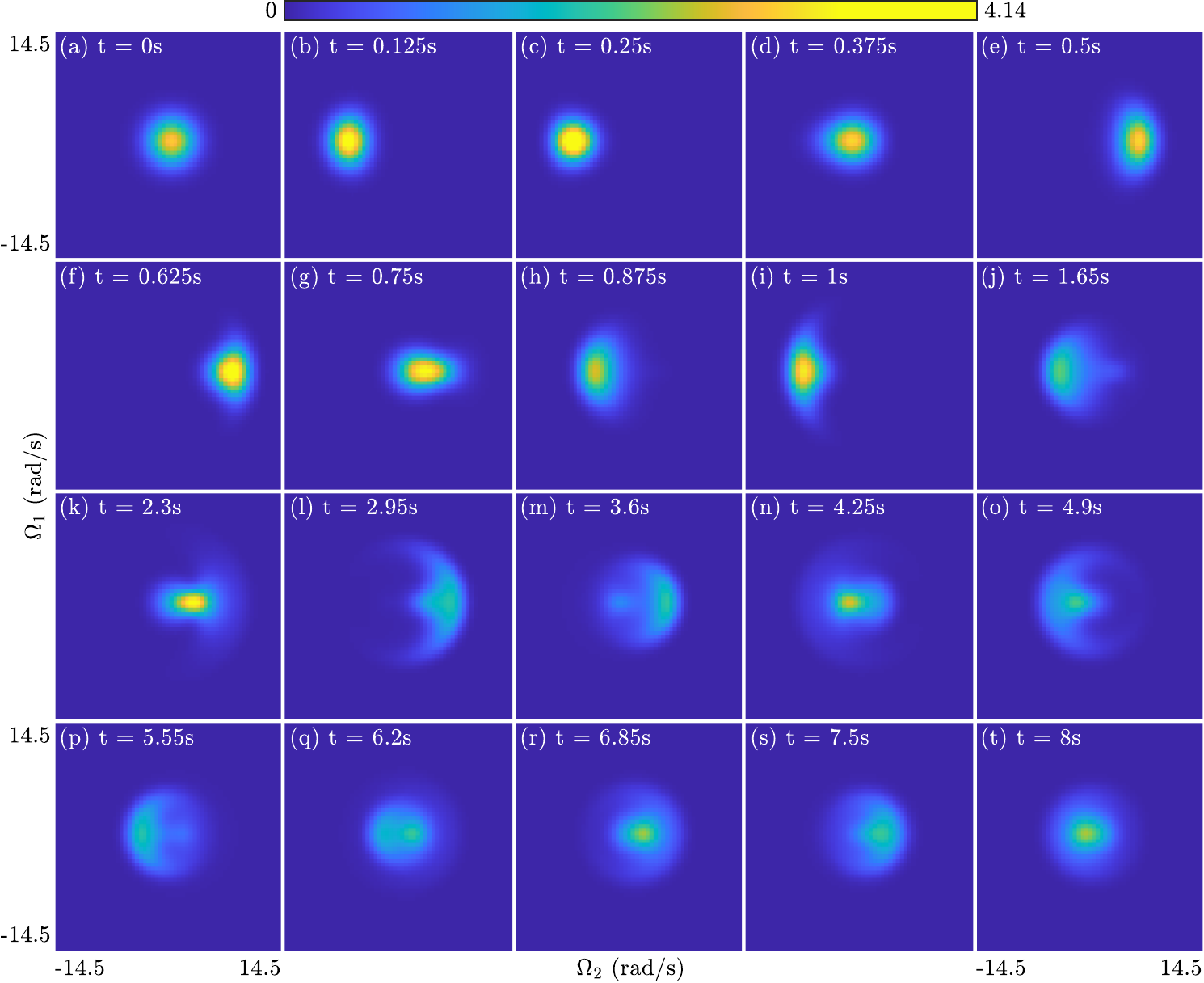}
	\caption{Marginal distribution of angular velocity without collisions. \label{fig:Omega_cont}}
\end{figure}

The marginal density for the angular velocity is shown in \cref{fig:Omega_cont}.
Initially, it is concentrated around zero (\cref{fig:Omega_cont}(a)).
After the pendulum is released, the angular velocity around $\vec b_2$, i.e., $\Omega_2$ accelerates and decelerates in the negative direction (\cref{fig:Omega_cont}(b)-(d)), as the pendulum swings and reaches the opposite limit.
Then $\Omega_2$ accelerates into the positive direction as the pendulum swings back, and this is repeated.
Finally, $\Omega$ becomes concentrated near zero again, after the energy is mostly dissipated (\cref{fig:Omega_cont}(t)).
During the process, the distribution of angular velocity displays some interesting shapes (\cref{fig:Omega_cont}(i)-(o)).
This is because the area with smaller $\Omega_1$ leads to oscillations compared to other areas with larger $\Omega_1$.
This illustrates one of the benefits of the proposed method that is capable of representing an arbitrary density function, and this cannot be achieved with the common Gaussian distribution. 

Next, the numerical results of the proposed method are compared against a Monte Carlo simulation with one million samples in \cref{fig:compare}.
The differences of the mean attitude and the $b_3$ direction between the Monte Carlo simulation and the proposed method with $l_0=n_0=20,30$ are shown in \cref{fig:compare}(a)(b), in terms of angles. 
It is shown that the difference of the mean direction of $\vec b_3$ is below \SI{1.5}{\deg}, which is very small compared to the standard deviation of attitude that is around \SI{15}{\deg}.
The difference of the mean attitude becomes relatively large (above \SI{5}{\deg}) after \SI{6}{\second}.
However, this is contributed by the fact that the rotation around $\vec b_3$ becomes close to a uniform distribution as seen in \cref{fig:R_cont}(q)-(t), thereby making the mean value less distinctive.
The standard deviations of attitude around $\vec e_1$ and $\vec e_2$ axes are compared in \cref{fig:compare}(c), with their discrepancies  more explicitly shown in \cref{fig:compare}(d).
Again the differences between the Monte Carlo simulation and the proposed method are small compared to their absolute magnitudes.
Similarly, the mean and standard deviation of angular velocity are compared in \cref{fig:compare}(e)(g), with the differences depicted in \cref{fig:compare}(f)(h).
In general, the larger bandwidth $b_0=n_0=30$ makes the uncertainty propagation more accurate compared to $b_0=n_0=20$, especially for the dispersion represented by standard deviations.
The moments computed by the proposed method are consistent with the Monte Carlo simulation.
However, the proposed method provides the probability distribution that carries the complete stochastic properties of the hybrid state beyond the moments. 

\paragraph{Propagation of GSHS}
Next, we propagate the uncertainty of the pendulum with collisions, i.e., both \eqref{eqn:FP_cont} and \eqref{eqn:FP_dist} are integrated.
The marginal distribution of attitude is shown in \cref{fig:R_hybrid}, and in \cref{fig:b3_hybrid} where only the marginal distribution of $\vec b_3$ is observed from bottom.
The wall is depicted by a gray plane.
Similar to the case without collisions, initially the attitude is concentrated where $\vec b_3$ is \SI{60}{\deg} from the vertical.
And after the pendulum is released, it swings about the $\vec e_2$ axis.
When the pendulum collides with the wall (\cref{fig:R_hybrid}(c)), it cannot penetrate through the wall, but rebounds backwards (\cref{fig:R_hybrid}(d)).
This is more clearly seen by comparing \cref{fig:b3_hybrid}(c) with \cref{fig:b3_cont}(c).
Then the pendulum swings back to somewhere below the initial position (\cref{fig:R_hybrid}(e)) due to the friction.
These are repeated for several cycles until the energy is mostly dissipated, when $\vec b_3$ is concentrated around the vertical direction, and no longer reaches the wall (\cref{fig:R_hybrid}(r)-(t)).
Compared with the case without collisions, the energy is dissipated more quickly, since it is also lost during the collision due to the coefficient of restitution less than one, besides the damping.
Note that there are some densities of $b_3$ that are slightly on the left of the wall during the collision (\cref{fig:b3_hybrid}(c),(h)).
This is because the rate function \eqref{eqn:pendulum_lambda} is not infinitely large in the guard set, thus there is a small probability that the discrete jump is not triggered when the pendulum is on the left.
But this probability becomes smaller when the pendulum further penetrates through the wall, since the rate function increases.
This can be interpreted as that the probability of rebounds increases as the third body-fixed axis $\vec b_3$ becomes closer to the wall.
The computational benefit is that the density changes gradually around the boundary of the guard set, instead of becoming zero abruptly like a step function, which allows capturing the space variation of the density function without excessively high bandwidth.

The marginal density of angular velocity is shown in \cref{fig:Omega_hybrid}.
Similar to the case without collisions, the angular velocity is concentrated around zero initially (\cref{fig:Omega_hybrid}(a)), and accelerates into the negative $\Omega_2$ direction after the pendulum is released (\cref{fig:Omega_hybrid}(b)).
Nevertheless, instead of decelerating to zero, the angular velocity undergoes jump due to the collision, i.e., the negative $\Omega_2$ is reset to be positive instantly during the discrete jump (\cref{fig:Omega_hybrid}(c)) according to the reset kernel \eqref{eqn:pendulum_kappa}, thereby separating the angular velocity distribution into two parts. 
Later, most of the angular velocity has completed the collision and $\Omega_2$ continuous to accelerate (\cref{fig:Omega_hybrid}(d)) until the pendulum reaches the vertical position, and begins to decelerate (\cref{fig:Omega_hybrid}(e)) afterwards.
These are repeated by several cycles, until the energy is mostly dissipated and the angular velocity is concentrated around zero again (\cref{fig:Omega_hybrid}(r)-(t)).
These illustrate that the proposed method successfully captures the complex interplay between the uncertainty distributions of attitude and angular velocity, as well as the collision, while generating the propagated density function for the hybrid state. 

The propagated uncertainty using the proposed method with $l_0=n_0=30$ is also compared with a Monte Carlo simulation with a million samples in \cref{fig:compare_hybrid}.
It is seen the differences of mean attitude and mean direction of $\vec{b}_3$ are small compared with the attitude standard deviation. 
The differences of attitude standard deviation, mean and standard deviation of angular velocity are in general within 10\% of their absolute magnitudes.

\begin{figure}
	\centering
	\begin{subfigure}{0.49\textwidth}
		\centering
		\includegraphics{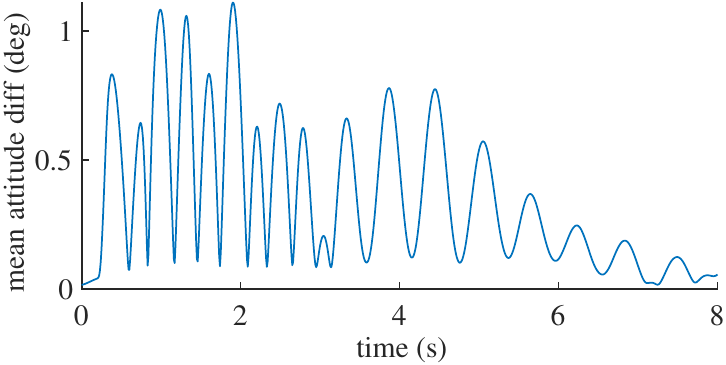}
		\caption{Difference of mean attitude \label{subfig:R_diff_hybrid}}
	\end{subfigure}
	\begin{subfigure}{0.49\textwidth}
		\centering
		\includegraphics{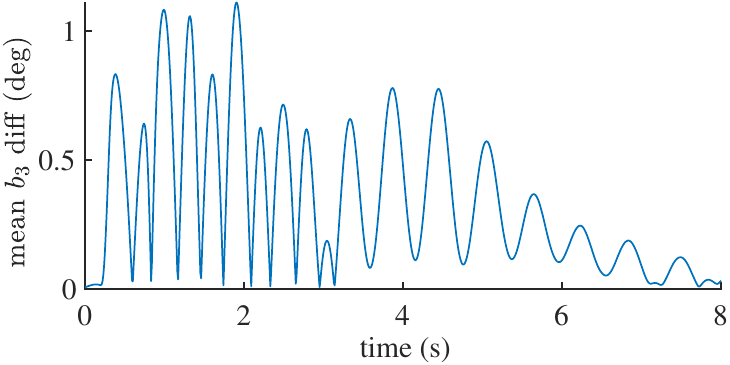}
		\caption{Difference of the mean of $b_3$ \label{subfig:b3_diff_hybrid}}
	\end{subfigure}
	\begin{subfigure}{0.49\textwidth}
		\centering
		\includegraphics{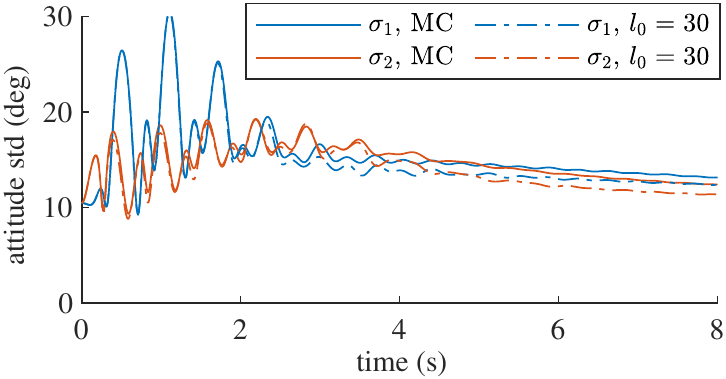}
		\caption{Attitude dispersion \label{subfig:std_R_hybrid}}
	\end{subfigure}
	\begin{subfigure}{0.49\textwidth}
		\centering
		\includegraphics{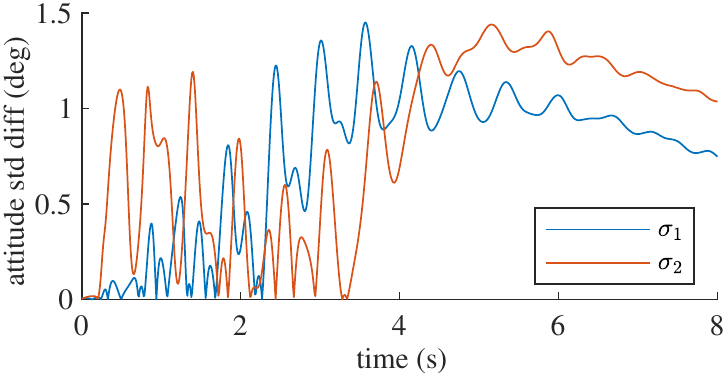}
		\caption{Difference of attitude dispersion \label{subfig:std_Rdiff_hybrid}}
	\end{subfigure}
	\begin{subfigure}{0.49\textwidth}
		\centering
		\includegraphics{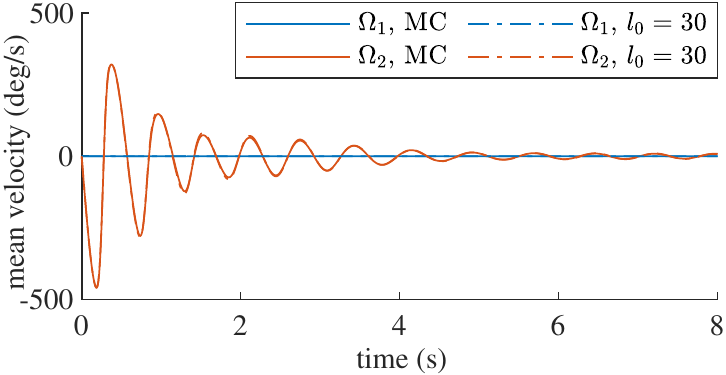}
		\caption{Mean angular velocity \label{subfig:mean_Omega_hybrid}}
	\end{subfigure}
	\begin{subfigure}{0.49\textwidth}
		\centering
		\includegraphics{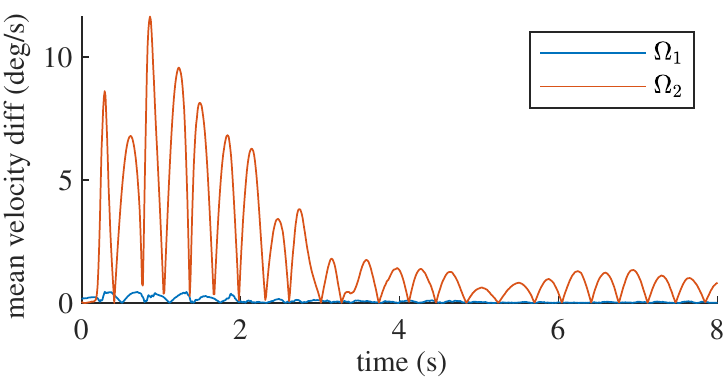}
		\caption{Difference of mean angular velocity \label{subfig:mean_Omegadiff_hybrid}}
	\end{subfigure}
	\begin{subfigure}{0.49\textwidth}
		\centering
		\includegraphics{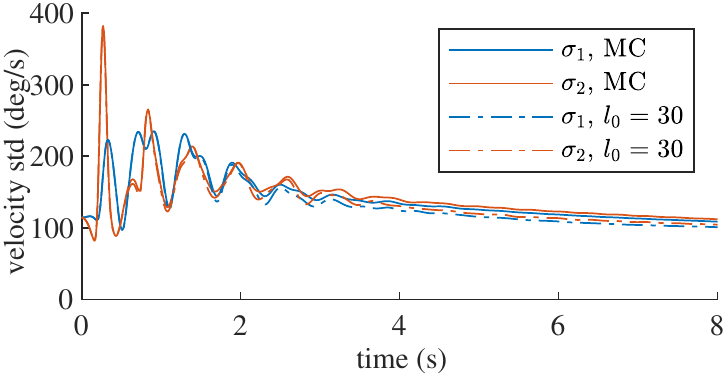}
		\caption{Angular velocity dispersion \label{subfig:std_Omega_hybrid}}
	\end{subfigure}
	\begin{subfigure}{0.49\textwidth}
		\centering
		\includegraphics{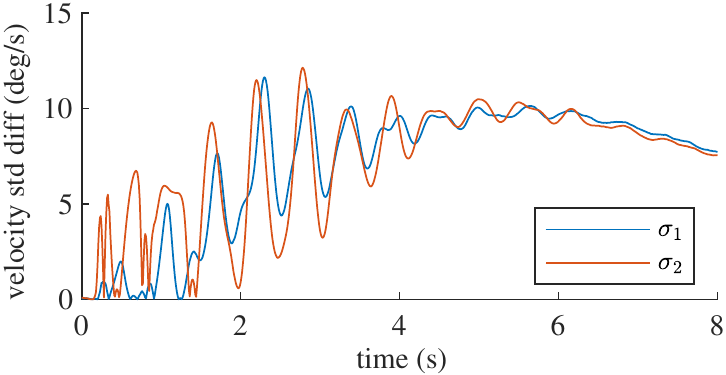}
		\caption{Difference of angular velocity dispersion \label{subfig:std_Omegadiff_hybrid}}
	\end{subfigure}
	\caption{Comparison of proposed method with Monte Carlo simulation with collisions. \label{fig:compare_hybrid}}
\end{figure}

\begin{figure}
	\centering
	\includegraphics{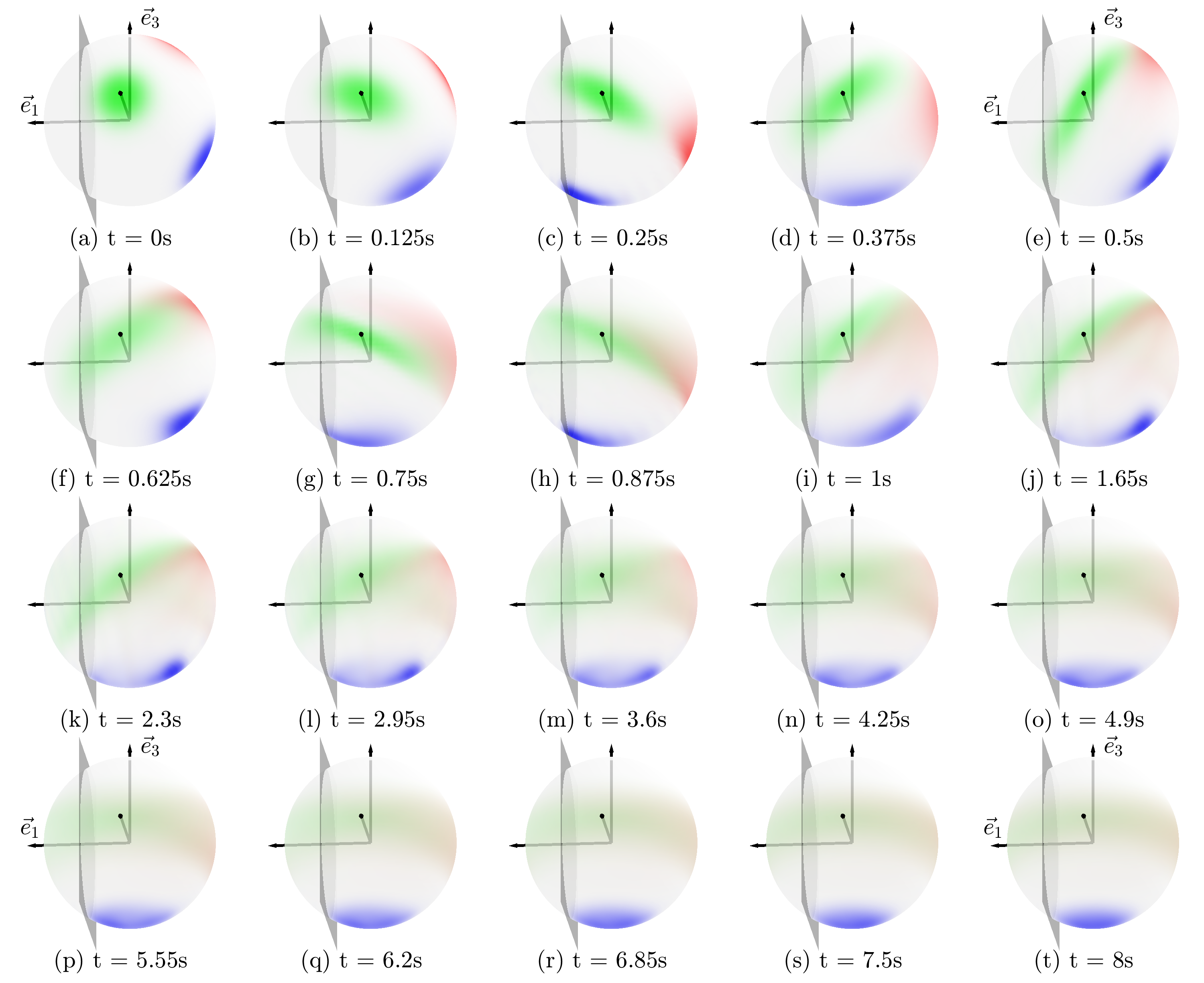}
	\caption{Marginal distribution of attitude with collisions. \label{fig:R_hybrid}}
\end{figure}

\begin{figure}
	\centering
	\includegraphics{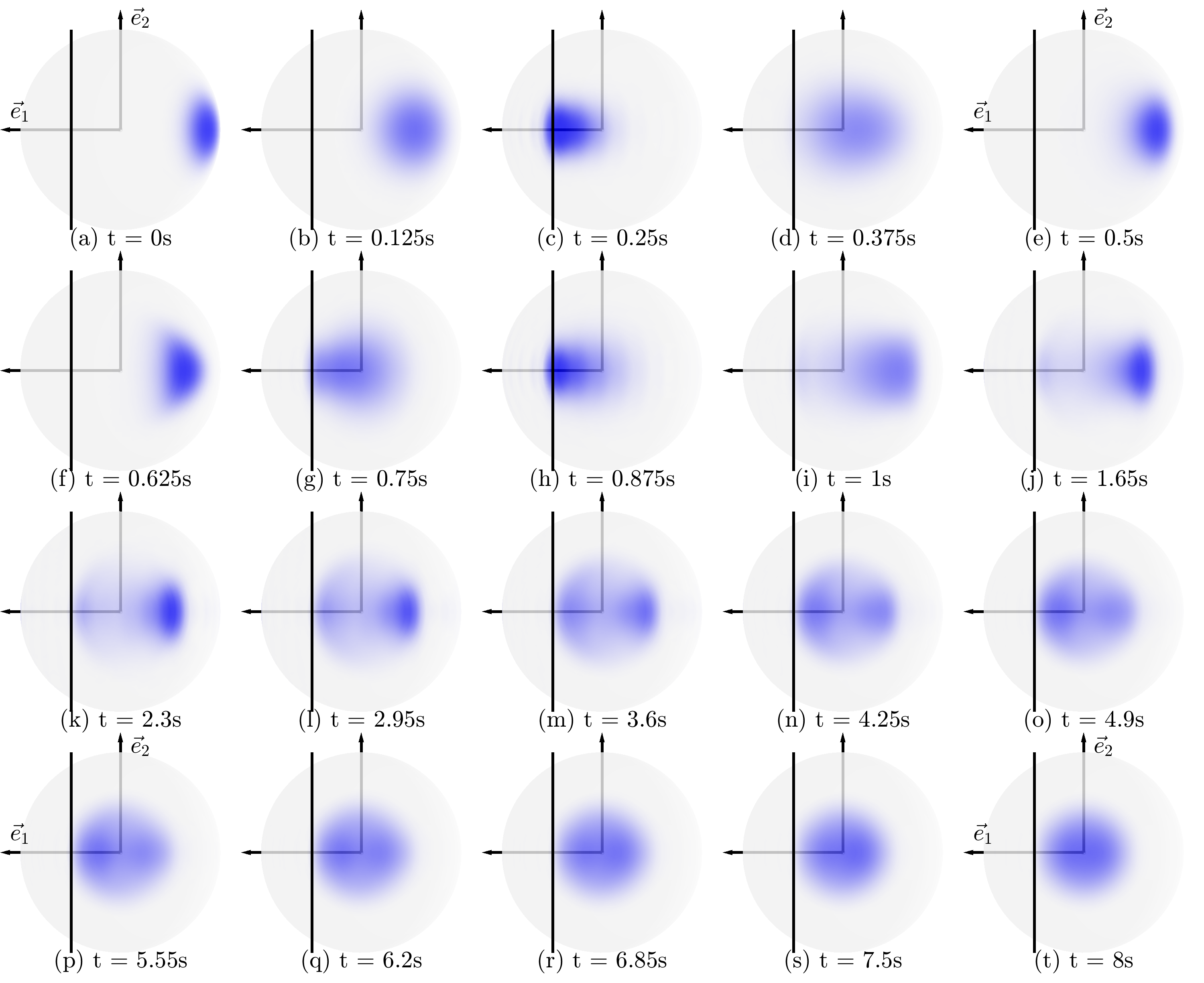}
	\caption{Marginal distribution of $b_3$ with collisions. \label{fig:b3_hybrid}}
\end{figure}

\begin{figure}
	\centering
	\includegraphics{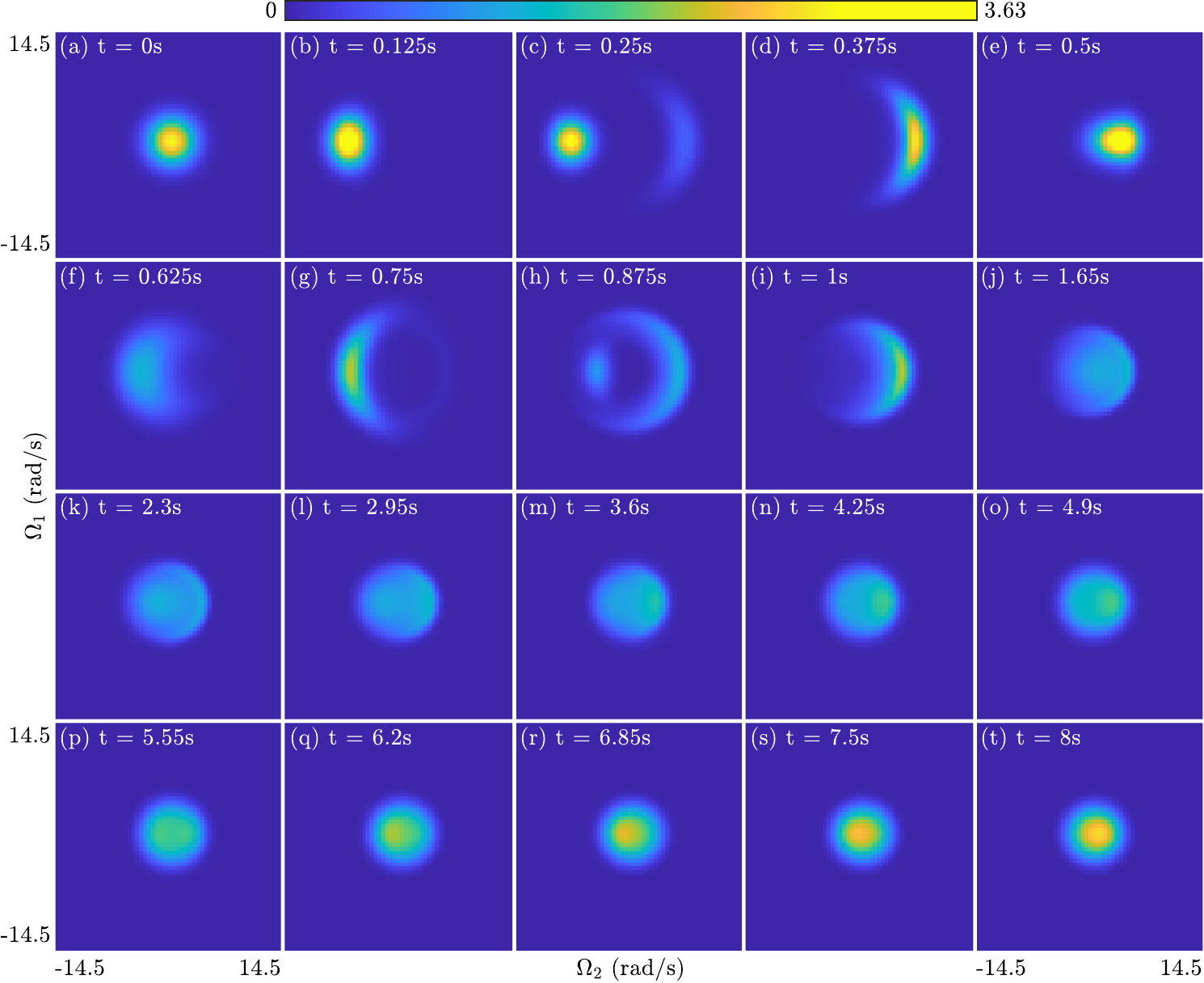}
	\caption{Marginal distribution of angular velocity with collisions. \label{fig:Omega_hybrid}}
\end{figure}


\section{Conclusions} \label{sec:conclusion}
In this paper, we propose a computational framework to propagate the uncertainty of a general stochastic hybrid system where the continuous state space is a compact Lie group.
The Fokker-Planck equation for the GSHS is split into two parts: a partial differential equation corresponding to the continuous dynamics, and an integral equation corresponding to the discrete dynamics.
The two split equations are solved alternatively and combined using a first order splitting scheme.
In particular, the PDE is solved using the classic spectral method, by invoking noncommutative harmonic analysis on a compact Lie group.
The proposed method is applied to a 3D pendulum that collides with a planar wall.
It is exhibited that the proposed method is able to capture complex uncertainty distributions with arbitrary shapes or large dispersion, and the computed density function can be directly used for visualization or for constructing any stochastic properties. 


\bibliographystyle{siamplain}
\bibliography{SIADS21}

\end{document}